\documentclass[10pt]{amsart}
\usepackage[english]{babel}
\usepackage{amssymb}
\usepackage{hyperref}
\usepackage[lite, initials]{amsrefs}
\usepackage{tikz-cd}
\usepackage{tikz}
\usetikzlibrary{intersections, decorations.markings, decorations.pathmorphing}
\usepackage{graphicx}
\usepackage{amssymb, mathrsfs, enumerate}
\usepackage{xcolor}
\usepackage{comment}
\usepackage{relsize}
\usepackage{bbm}
\newcommand{\bbN}{{\mathbb N}}
\newcommand{\bbQ}{{\mathbb Q}}
\newcommand{\bbR}{{\mathbb R}}

\newcommand{\bbZ}{{\mathbb Z}}
\newcommand{\bbC}{{\mathbb C}}

\newcommand{\bbF}{\mathbb{F}}

\newcommand{\bfG}{{\mathbf G}}

\newcommand{\calF}{\mathcal{F}}

\newcommand{\calO}{\mathcal{O}}

\newcommand{\banach}{\mathsf{Ban}} 
\newcommand{\defq}{\mathrel{\mathop:}=}
\newcommand{\qdef}{=\mathrel{\mathop:}}

\newcommand{\id}{\operatorname{id}}
\newcommand{\im}{\operatorname{im}}

\newcommand{\res}{\operatorname{res}}

\newcommand{\vol}{\operatorname{vol}}

\newcommand{\rig}{\operatorname{rig}}

\newcommand{\Prob}{\operatorname{Prob}}

\newcommand{\coker}{\operatorname{coker}}

\newcommand{\SL}{\operatorname{SL}}

\newcommand{\Aut}{\operatorname{Aut}}

\newcommand{\Sub}{\operatorname{Sub}}

\newcommand{\Sp}{\operatorname{Sp}}

\newcommand{\tors}{\operatorname{tors}}

\newcommand{\temp}{\operatorname{temp}}

\newcommand{\acts}{\curvearrowright}

\newcommand{\norm}[1]{{\left\lVert #1\right\rVert}}

\let\origbullet\bullet
\renewcommand{\bullet}{{\mathbin{\mathsmaller{\origbullet}}}}

\newtheorem{theorem}{Theorem}
\newtheorem{lemma}[theorem]{Lemma}

\newtheorem{cor}[theorem]{Corollary}

\newtheorem{prop}[theorem]{Proposition}

\newtheorem{rem}[theorem]{Remark}
\theoremstyle{definition}

\newtheorem*{theorem_o}{Theorem}
\newtheorem{defn}[theorem]{Definition}
\newtheorem{convention}[theorem]{Convention}

\newtheorem{example}[theorem]{Example}
\newtheorem{question}[theorem]{Question}
\newtheorem{remark}[theorem]{Remark}
\newtheorem{conjecture}[theorem]{Conjecture}

\begin{document}

\title[Higher property~T]{Higher property~T and below-rank phenomena of lattices}

\subjclass[2020]{Primary 22D10, 22E40; Secondary 20J06, 20F65, 46L10}
\keywords{higher property T, group cohomology, lattices, operator algebras}
\author{Uri Bader}
\address{UMD}
\email{uri.bader@gmail.com}

\author{Roman Sauer}
\address{Karlsruhe Institute of Technology}
\email{roman.sauer@kit.edu}

\begin{abstract}
The purpose of this paper is twofold. We explore higher property~T as an abstract group-theoretic property. In particular, we provide new operator-algebraic characterizations of higher property~T. 
Then we turn to lattices in semisimple Lie groups. We relate higher property~T to other cohomological, rigidity and geometric phenomena below the real rank. The second part outlines a conjectural framework that unifies these aspects and reviews recent advances.
\end{abstract}
\maketitle

%\tableofcontents

\section{Introduction}
The purpose of this paper is twofold:  
To explore higher property~T, a higher analogue of Kazhdan's property~T, as a group-theoretic property and to relate it to various  ``below-rank phenomena'' of semisimple groups and their lattices\footnote{~A lattice $\Gamma$ in a locally compact group $G$ is a discrete subgroup such that $\vol(G/\Gamma)<\infty$.}.

The \emph{rank} of a semisimple algebraic group over a field $F$ is the dimension of a maximal $F$-split torus in it.
Over an archimedean local field, that is for a Lie group, it coincides with the dimension of a maximal flat in the associated symmetric space,
and over a non-archimedean local field it coincides with the dimension of the associated Bruhat-Tits building.
It is an empirical fact that this invariant is the threshold for various rigidity phenomena associated with the group $G=\mathbf{G}(F)$, as well as with its lattices.

Our main focus is on the following prototypical example, saying that lattices in simple groups satisfy a higher version of Kazhdan's property~T.

\begin{theorem} \label{thm:higher}
       Let $\mathbf{G}$ be a simple algebraic group of rank $r$ over a characteristic 0 local field $F$ and let $\Gamma<G=\mathbf{G}(F)$ be a lattice. 
       Then $\Gamma$ has property $(T_{r-1})$. If the field $F$ is non-archimedean then $\Gamma$ has property $[T_{r-1}]$.
\end{theorem}

To use cohomology in defining higher property~T seems natural in view of the Guichardet-Delorme characterization of classical property~T. 

\begin{defn}[{\cite[Definition 1.1]{badsau}}] \label{def:higherT}
    A locally compact second countable group $G$ has property $[T_n]$ if
its continuous cohomology $H_c^j(G,V)$ vanishes for every unitary $G$-representation $V$
and every $1 \leq j \leq n$. It has property $(T_n)$ if the same is true provided that $V$ has
no nontrivial $G$-invariant vectors.
\end{defn}

Property $(T_1)$ is equivalent to property $[T_1]$, and both properties are equivalent to Kazhdan’s property T.
The fact that higher rank lattices\footnote{~Higher rank lattices are lattices in groups of rank $r\geq 2$.} have property T is due to Kazhdan.
For $F$ non-archimedean Theorem~\ref{thm:higher} is due to Dymara-Januszkiewicz \cite{Dy-Jan} and
Grinbaum-Reizis-Oppenheim \cite{GR-Opp}
 (using the so called ``Garland Method'').
The general case is 
due to the authors~\cite{badsau}. The proof will be sketched in \S\ref{sec:proof}.

\begin{rem}\label{rem: harder}
  In view of \cite{harder}, it is highly likely that the characteristic 0 assumption on $F$ in Theorem~\ref{thm:higher} could be removed.
  See \S\ref{subsec:poly} for a further discussion.
\end{rem}

\begin{rem}
    Over an archimedean field~$F$, the groups $G$ and lattices $\Gamma$ as in Theorem~\ref{thm:higher} do not typically satisfy $[T_{r-1}]$.
    For example $H^2(\Sp_n(\mathbb{Z}),\mathbb{C})\neq 0$.
\end{rem}

Despite the previous remark, the range of higher property~T can exceed the rank. There are even rank~$1$ groups with higher property~T, as the following theorem demonstrates. 
Recall that isometry group of the Cayley hyperbolic plane, denoted $F_4^{-20}$, is a rank 1 simple Lie group.

\begin{theorem}\label{thm:F4}
The group $F_4^{-20}$ and its lattices have property~$[T_3]$.
\end{theorem}

Noting that the 2nd and 3rd cohomologies for $F_4^{-20}$ with trivial coefficients vanish, see~e.g \cite{Borel-Hirzebruch}*{\S19.1}, Theorem~\ref{thm:F4} follows for the Lie group and its uniform lattices by \cite{badsau}*{Theorem~B}. For non-uniform lattices it is due to L{\'o}pez Neumann and Paucar~\cite{neumann+paucar}*{Corollary~0.3}.

\begin{rem}
    Over a non-archimedean field, this is never the case, as the Steinberg representation is always cohomological in the degree of the rank.
\end{rem}

Theorem~\ref{thm:higher} might be conjecturally extended in two ways: By extending the class of coefficients to Banach spaces and by allowing semisimple or $S$-semisimple groups instead of only simple groups. 
\begin{conjecture} \label{conj:SR}
   Let $\Gamma$ and $r$ be as in Theorem~\ref{thm:higher}.
   Let $V$ be a super-reflexive Banach space\footnote{~A Banach space is super-reflexive if any of its ultrapowers is reflexive.} and $\Gamma\to B(V)$ be a linear isometric representation with no invariants, that is $V^\Gamma=0$.
   Then for every $j<r$, 
   \[ H^j(\Gamma,V)=0. \]
\end{conjecture}

In short, we regard this conjecture as saying that $\Gamma$ has the \emph{super-reflexive property $(T_{r-1})$}.
The case $j=1$ (which was conjectured in \cite{BFGM}) was proved recently in \cite{Oppenheim} and \cite{Sa-Le-at}, see Theorem~\ref{thm:Opp} below.
When removing the assumption $V^\Gamma=0$ but leaving as is the rest of the setup of Conjecture~\ref{conj:SR},
using induction techniques developed in \cite{badsau}, the conjecture implies that for every $j<r$,
we have
\begin{equation} \label{eq:simple}
    H^j(\Gamma,V) \cong H^j(\Gamma,V^\Gamma) \cong H^j(G,V^\Gamma) \cong V^\Gamma \otimes H^j(G,\mathbb{C}),
\end{equation}
   where the left map comes from the inclusion $V^\Gamma\to V$ and the middle map from the restriction from $G$ to $\Gamma$.

   Next, we discuss the case where the group $G$ is not necessarily simple but semisimple or $S$-semisimple\footnote{~$G=\prod_{s\in S}\mathbf{G}_s(F_s)$ for local fields $F_s$ and $F_s$-semisimple groups $\mathbf{G}_s$, $r$ being the sum of ranks.}.
   Then there is no hope for the vanishing below the rank phenomenon.
   However, we still expect that all the cohomology of~$\Gamma$ below the rank will come from $G$, as in \eqref{eq:simple}.

   \begin{conjecture} \label{conj:semisimple}
       Let $G$ be an $S$-semisimple group of rank $r$ and $\Gamma<G$ an irreducible\footnote{~$\Gamma$ is irreducible if its projection to each simple factor of $G$ is dense.} lattice.
       Let $V$ be a super-reflexive Banach space and $\Gamma\to B(V)$ be a linear isometric representation.
       There exists a $\Gamma$-invariant subspace $V_0\subset V$ on which the $\Gamma$-action extends to a $G$-action and such that for every $j<r$, 
   \[ H^j(\Gamma,V) \cong H^j(\Gamma,V_0) \cong H^j(G,V_0), \]
   where the left map comes from the inclusion $V_0\hookrightarrow V$ and the right map from the restriction from $G$ to $\Gamma$.
   In particular, if
   the $\Gamma$-action on $V$ does not extend to a $G$-action on any non-trivial subrepresentation then $H^j(\Gamma,V)= 0$.
   \end{conjecture}

    The following two theorems from \cite{badsau} establish Conjecture~\ref{conj:semisimple} in two important special cases. Theorem~\ref{thm:semisimple Lie without T} is a generalization of Borel's stability theorem to a wider range of degrees and from constant to unitary coefficients. 

\begin{theorem}[\cite{badsau}*{Theorem D}]\label{thm:semisimple Lie with T}   
If $G$ is a semisimple Lie group with property~T and $V$ is a unitary $\Gamma$-representation, then Conjecture~\ref{conj:semisimple} holds true.
\end{theorem}

\begin{theorem}[\cite{badsau}*{Theorem C}]\label{thm:semisimple Lie without T}   
If $G$ is a semisimple Lie group (not necessarily with property~T) and $V$ is the restriction of a unitary $G$-representation, then Conjecture~\ref{conj:semisimple} holds true.
\end{theorem}

We expect that the proof of Theorem~\ref{thm:higher} could be also extended to the S-semisimple case (Conjecture~\ref{conj: S-arithmetic filling} below), which is concerned with an extension of Theorem~\ref{thm: leuzinger+young} by Leuzinger-Young to this setting.
This is due to the Strong Spectral Gap Conjecture~\ref{conj:SSG}, 
which will be discussed in \S\ref{sec:deg1}.
We stress that Conjecture~\ref{conj:semisimple} is open even in degree~$1$.
The degree~$1$ vanishing over unitary representations, which is relevant for irreducible lattices in product of rank~$1$ groups, is quite intriguing.
It is equivalent to Shalom's Spectral Gap Conjecture~\ref{conj:SG}.
Its validity implies for instance the full solution of Connes' \emph{Character Rigidity conjecture} and Lubotzky's \emph{property $\tau$ conjecture}. See \S\ref{sec:deg1} for details.

In a paper \cite{BBBS}, yet to be written together with Saar Bader and  Shaked Bader, we show that lattices satisfy vanishing results for general isometric actions on von Neumann modules.
In \S\ref{sec:proof} we will sketch the proof of the next theorem 
and relate it to some further conjectures. 

For the notions of \emph{$\Gamma$-unitary von Neumann algebra} and the \emph{noncommutative $L^p$-space $L^p(M)$ of a von Neumann algebra~$M$} we refer to the beginning of~\S\ref{subsec: vN algebras}. The reader who prefers a von-Neumann free statement should look at Corollary~\ref{cor:com} below. 

\begin{theorem}[\cite{BBBS}] \label{thm:noncom}
Let $\Gamma$ be a lattice in a simple Lie group $G$ of rank $r$ and $M$ a von Neumann algebra. Then for $k\leq r-1$ and $1\leq p < \infty$, for every 
linear isometric action of $\Gamma$ on $L^p(M)$
the inclusion $L^p(M)^\Gamma\hookrightarrow L^p(M)$ induces isomorphisms
    \[ H^k(\Gamma, L^p(M)^\Gamma) \cong H^k(\Gamma,L^p(M)). \]
    Moreover, $H^r(\Gamma,L^p(M))$ is Hausdorff.
    For $p=\infty$ this holds under the extra assumption that $M$ is $\Gamma$-unitary.
    Similar statements hold for $G$.
\end{theorem}

A duality argument shows that the homological analogue of the above results holds for $1< p \leq \infty$ and for $p=1$ under the extra assumption that $M$ is $\Gamma$-unitary.
However, the extra assumption in the theorem is indeed needed for $p=\infty$ and for the homological version with $p=1$.
Compare with the remark after Conjecture~\ref{conj:tracial}. 
For its importance, we state the corollary of this theorem in case $M$ is a commutative von Neumann algebra.

\begin{cor} \label{cor:com}
Let $\Gamma$ be a lattice in a simple Lie group $G$ of rank $r$. Then for $k\leq r-1$ and $1\leq p < \infty$, for every 
linear isometric action of $\Gamma$ on $L^p(X)$ where $X$ is a measured space,
the inclusion $L^p(X)^\Gamma\hookrightarrow L^p(X)$ induces isomorphisms
    \[ H^k(\Gamma, L^p(X)^\Gamma) \cong H^k(\Gamma,L^p(X)). \]
    In particular, if $L^p(X)^\Gamma=0$, then for every $k\leq r-1$, $H^k(\Gamma,L^p(X))=0$.
        Similar statements hold for $G$.
\end{cor}

\begin{rem}\label{rem: gaussian}
Associated with any unitary representation $V$ of $\Gamma$ there is a corresponding Gaussian measure, which yields a probability measure preserving action of $\Gamma$ on a space $X$ such that $V$ embeds into $L^2(X)$ as a subrepresentation~\cite{zimmer}*{Example~5.2.13 on p.~110}.
\end{rem}

\begin{rem}[Comparision of the paper~\cite{badsau} and the forthcoming paper~\cite{BBBS}]
    We clarify the relation between the original higher property T paper~\cite{badsau} and the announced paper~\cite{BBBS}.
In view of Remark~\ref{rem: gaussian}
    we could express Theorem~\ref{thm:higher}, that is~\cite{badsau}*{Theorem~A}, in a similar way as Corollary~\ref{cor:com} in the case~$p=2$. On the other hand, the generalization of Theorems B,C,D,E,F,G in~\cite{badsau}, which includes Theorems~\ref{thm:F4},~\ref{thm:semisimple Lie with T} and~\ref{thm:semisimple Lie without T}, to the setting of Banach representations remains unclear at the moment. 
    
    The forthcoming proofs of Theorem~\ref{thm:noncom} and Corollary~\ref{cor:com} do not subsume the proof of Theorem~\ref{thm:higher} but depend on it.    
    The reason is that a comparatively simple step in~\cite{badsau} is no longer applicable for $p\ne 2$ and must be replaced by a completely new and delicate analysis, while the full technical machinery developed in~\cite{badsau} remains indispensable. See also Remark~\ref{rem: proof strategy}.
\end{rem}
We denote by $(T_n)_{L^p}$, where $n=r-1$, the property established for $\Gamma$ in this corollary.
Note that $(T_n)_{L^2}$ is simply $(T_n)$.
We will see later that $(T_n)_{L^1}$ has some particular geometric applications.

%The proof of Theorem~\ref{thm:noncom}, along with some related conjectures, will be discussed in section \ref{sec:proof}.

The previous theorems and conjectures have lots of interesting applications and ramifications in relation to various rigidity and geometric phenomena below the rank. We discuss these in~\S\ref{subsec: cohomological below-rank},~\S\ref{sec:geo} and~\S\ref{sec:sd}. Let us highlight a few of them here.

In~\S\ref{subsec: Borel stability} we discuss generalizations of Borel's stability theorem and theorems of Borel-Yang to a wider range of degrees and to unitary coefficients. These are based on Theorem~\ref{thm:semisimple Lie without T} and certain adelic generalizations of it. 
Corollary~\ref{cor:com} implies conjectures of Gromov and Farb that will be dealt with in \S\ref{sec:LP} and \S\ref{sec:geo}, respectively.
In \S\ref{sec:other} we speculate on torsion growth,
a speculation which we relate by the end of \S\ref{sec:geo}, after discussing expansions and waist inequalities,
to the relations between expansion over the reals, over the integers and over finite fields.
In \S\ref{sec:sd} we regard cohomological vanishing in small degrees. 
We discuss degree 1 applications of Conjecture~\ref{conj:semisimple}, which we already alluded to earlier. Then we turn to degree 2 applications of Theorem~\ref{thm:higher} to the theory of stability.
Finally, we briefly discuss an interesting use of Theorem~\ref{thm:F4} in~\S\ref{subsec:measdiv}

We are not shy of making here many conjectures and speculations.
For the reader's orientation, we included Figure~\ref{fig1}, collecting various concepts discussed in \S\ref{sec:geo} and their possible relations, and Figure~\ref{fig2}, collecting various conjectures discussed in \S\ref{sec:sd} and their known graph of implications. 

While still essentially all examples of higher property~T arise from lattices in semisimple groups, the importance of its notion and its geometric implications (e.g.~Theorem~\ref{thm:Farb2}) suggest to study it more broadly as an abstract group-theoretic property. This is done in~\S\ref{sec: as abstract group property}. The style of that section is less survey-like than~\S\ref{sec: lattices in semisimple groups} and more research-oriented. We give a taste by presenting two results of this section here. 

We provide a sum-of-squares characterization of property~$(T_n)$ in~\S\ref{subsec: C-algebraic reformulation} for the Laplace-operator regarded as a matrix over the group ring, generalizing the celebrated Ozawa criterion \cite{OzawaT}.
See \S\ref{subsec:FP} for a discussion of finiteness properties.

\begin{theorem}\label{thm:grpalg}
Assume $\Gamma$ satisfies the finiteness property $FP_\infty(\bbQ)$.
     The group $\Gamma$ has property $(T_n)$ if and only if 
  for every $0\leq k \leq n$ there exists $\epsilon > 0$ and elements $x_1,\dots,x_j\in M_{m_k}(\bbC\Gamma)$ such that 
    \[ \Delta_0(\Delta_k -\epsilon)\Delta_0 = \sum_{i=0}^j x_i^*x_i.\]   
\end{theorem}
Here $\Delta_k$ is the Laplace operator in degree $k$, and $\Delta_0$ is an element in the group ring which, viewed as a scalar, can be multiplied with matrices over the group ring. 
The corresponding result for $[T_n]$ (Theorem~\ref{thm: bader-nowak for bracket T}) was (essentially) proved in~\cite{Bader-Nowak}. See the comment after Theorem~\ref{thm: bader-nowak for bracket T} in~\S\ref{subsec: group algebra reformulation}. 

We also have a characterization of higher property~T in terms of the maximal $C^*$-algebra $C^\ast\Gamma$ of a group~$\Gamma$ proved in~\S\ref{subsec: C-algebraic reformulation}. In the sequel, we denote the kernel of the augmentation homomorphism $C^\ast\Gamma\to\bbC$ by $C^\ast_0\Gamma$. See~\S\ref{subsec: kazhdan proj} for a detailed discussion of $C^\ast_0\Gamma$. 

\begin{theorem}[Higher property T criteria] \label{thm:criteria}
Assume $\Gamma$ satisfies the finiteness property $FP_\infty(\bbQ)$.
For $n>0$, the following are equivalent. 
\begin{enumerate}
    \item The group~$\Gamma$ has property $[T_n]$.
    \item $H^k(\Gamma,C^*\Gamma)=0$ for all $k\in\{1,\dots, n\}$ and $H^{n+1}(\Gamma, C^\ast\Gamma)$ is Hausdorff. 
    \item $H_k(\Gamma,C^*\Gamma)=0$ for all $k\in\{1,\dots, n\}$ and $H_0(\Gamma, C^\ast\Gamma)$ is Hausdorff. 
\end{enumerate}
A similar equivalence holds for $(T_n)$ where $C^\ast\Gamma$ replaced by $C_0^\ast\Gamma$.
\end{theorem}

For modules over the maximal $C^\ast$-algebra we obtain the following result, which is also proved in~\S\ref{subsec: C-algebraic reformulation}. 

\begin{theorem}\label{thm:TnCor}
Assume $\Gamma$ satisfies the finiteness property $FP_\infty(\bbQ)$.
    If $\Gamma$ has property $[T_n]$, 
    then 
for every $C^*\Gamma$-module $V$ and for every $1\leq k \leq n$, $H^n(\Gamma,V)=H_k(\Gamma,V)=0$ and if $V$ is a Banach $C^*\Gamma$-module then $H^{n+1}(\Gamma,V)$ is Hausdorff.

If $\Gamma$ has property $(T_n)$, then the same holds, upon replacing $C^*\Gamma$ by $C_0^*\Gamma$. 
In particular, if $\Gamma$ has property $(T_n)$, then for every unitary $\Gamma$-module $V$, $H^{n+1}(\Gamma,V)$ is Hausdorff.
\end{theorem}
%\tableofcontents

\subsection*{Acknowledgments}
We thank Lubotzky for lots of advice, suggestions and encouragement. We thank Francesco Fournier Facio, Mikael de la Salle 
and Antonio L{\'o}pez Neumann
for answering questions (mentioned in the text). 
We thank Alon Dogon, David Fisher, Michael Glasner, Yuval Gorfine, Izhar Oppenheim, Yehuda Shalom and Stefan Witzel for insightful remarks over a first draft of this paper.
In particular we thank Pierre Pansu and Raz Slutsky for a very thorough reading of this first draft and the many comments they made.
Finally, we thank Saar Bader and Shaked Bader for the most delightful collaboration~\cite{BBBS} and their many inputs. 

In the initial phase of this project, the second author was supported by the DFG (Deutsche Forschungsgemeinschaft), project number 441426599.
The second author would also like to thank the Isaac Newton Institute for Mathematical Sciences, Cambridge, for support and hospitality during the programme \emph{Operators, Graphs, Groups} where work on this paper was undertaken. This work was supported by EPSRC grant no EP/K032208/1.

\section{Higher property T as an abstract property of groups} \label{sec: as abstract group property}

In this section we study higher property T as a property of abstract groups. We discuss the role of finiteness properties, give alternative characterizations using operator algebras, and discuss permanence properties under group-theoretic constructions. 

Unlike the next section which focuses on lattices in semisimple groups and mainly reviews recent results, this section contains new results (Theorems~\ref{thm:criteria},~\ref{thm:TnCor} and~\ref{thm:grpalg}), results that have only been stated in special cases in the literature before (Lemmas~\ref{lem:lapinv},~\ref{lem: duality reduced cohomology}, and~\ref{lem:basic}) and results that give a fresh perspective to known facts (for example, Corollary~\ref{cor: categorical epi and cokernel} and \S\ref{subsec: kazhdan proj}).

For the most part, we assume a mild finiteness property on the groups under consideration. See Convention~\ref{conv:ft} below. 
The role of this assumption is discussed in \S\ref{subsec:FP} and \S\ref{subsec: back to fin}.

To discuss higher property T as an abstract group property makes only sense if there are enough examples. In all non-trivial examples the verification of classical property~T for a group is more than just a mere example -- it is usually a theorem. The same is even more true for higher property T. Essentially, the only known non-trivial examples of groups with higher property T are lattices in semisimple groups as in Theorem~\ref{thm:higher} above. 

Another class of groups that has a chance to have higher property T are automorphism groups of free groups.
The group $\Aut(F_n)$ has property~T for $n\ge 4$. 
This was proved by Nitsche~\cite{nitsche} for $n=4$, Kaluba-Nowak-Ozawa for $n=5$~\cite{kaluba+nowak+ozawa} and  
by Kaluba-Kielak-Nowak for $n\geq 6$~\cite{kaluba+kielak+nowak}.
Further, Galatius showed that the rational cohomology of $\Aut(F_n)$ stabilizes to that of the symmetric group~$\Sigma_n$. As a corollary one obtains that $H^k(\Aut(F_n),\bbQ)=0$ for $n\geq 2k+2$~\cite{galatius}*{Corollary~1.2}. Finally, $\Aut(F_n)$ has property $FA_{\lfloor\frac{2n}{3}\rfloor}$ by~\cite{bridson-helly}. According to Theorem~\ref{thm:Farb2} this would be a consequence of higher property~T for Banach spaces, more precisely of property~$(T_{\lfloor\frac{2n}{3}\rfloor})_{L^1}$. 

Given the evidence above, it is natural to ask the following question.
\begin{question}
    Let $k\in\bbN$. 
    Is there an (explicit) $n_k\in\bbN$ such that $\Aut(F_n)$ has property $[T_k]$ for $n\geq n_k$? 
\end{question}

\subsection{Finiteness properties} \label{subsec:FP}
Property~T implies finite generation, but not finite presentation. An example due to Yves Cornulier is presented in \cite{bekka+harpe+valette}*{Theorem~3.4.2 on p.~172}. The same example shows that property $[T_2]$ does not imply finite presentation either, which we discuss next. 

\begin{example} \label{ex:nonfp}
Let $n\in \mathbb{N}$, and let $p$ be a prime. 
The group $\Sp_n(\mathbb{Z}[1/p]) \ltimes \mathbb{Z}[1/p]^{2n}$ is not finitely presented by~\cite{bekka+harpe+valette}*{Theorem~3.4.2 on p.~172}. Similarly, one shows that the locally compact group $\Sp_n(\mathbb{Q}_p) \ltimes \mathbb{Q}_p^{2n}$ is not compactly presented. In forthcoming work it will be shown that both groups have property $(T_{n-1})$. 
Moreover, the first group has a central extension that has $[T_2]$~(cf.~\cite{bekka+harpe+valette}*{Lemma~3.4.3 on p.~173}). 
\end{example}

We say that a group has the finiteness property $FP_n(\bbQ)$ if there is a projective $\bbQ\Gamma$-resolution of the trivial module~$\bbQ$ that is finitely generated in degrees up to $n$.
The finite generation of $\Gamma$ is equivalent to having the finiteness property $FP_1(\bbQ)$.
Finite presentation implies, but it is not implied by, the finiteness property $FP_2(\bbQ)$.
In fact, it is not implied by $FP_\infty(\bbQ)$, which means $FP_n(\bbQ)$ for all $n$.
The countable groups in Example~\ref{ex:nonfp} are $FP_\infty(\bbQ)$. So the following question posed by David Kazhdan to us remains open. 

\begin{question} \label{q:fin}
    Does $(T_n)$ imply $FP_n(\bbQ)$?
\end{question}

\begin{example}
Consider Theorem~\ref{thm:higher} in the case that the field~$F$ is of positive characteristic and $\Gamma$ is non-uniform lattice (which is automatic, unless $G$ is of type $A_n$).
Then $\Gamma$ (presumably) has $[T_{r-1}]$ (see Remark~\ref{rem: harder}), and it is of type $FP_{r-1}(\bbZ)$, in particular of type $FP_{r-1}(\bbQ)$, by the rank theorem of Bux-K\"ohl-Witzel~\cite{bux+koehl+witzel}. However, the group $\Gamma$ is not of type $FP_{r}(\bbZ)$ by a theorem of Bux-Wortman~\cite{bux+wortman}*{Theorem~1.2}. Very likely, the same holds with $\bbZ$ replaced by $\bbQ$.
\end{example}

\begin{convention} \label{conv:ft}
From now on, throughout this section, 
we assume that $\Gamma$ is a group that has the finiteness property $FP_\infty(\bbQ)$. For most purposes below, this assumption could be weakened to $FP_n(\bbQ)$ for $n$ high enough.
\end{convention}

\subsection{The role of the Laplace operators} \label{subsec: laplace operator}

Consider a based $\bbQ\Gamma$-resolution of $\bbQ$ by finitely generated free left $\bbQ\Gamma$-modules and let $P_\bullet$ be its tensor with $\bbC$,
\begin{equation}\label{eq: projective resolution} 0 \leftarrow \mathbb{C} \leftarrow P_0 \xleftarrow{d_1} P_1 \xleftarrow{d_2} P_2 \xleftarrow{d_3} \cdots  
\end{equation}
Each $P_n$ is equipped with a $\bbC\Gamma$-basis defined over $\bbQ$
and we assume as we may $P_0=\bbC\Gamma$.
We consider the cochain complex 
\[P^\bullet=\hom_{\bbC\Gamma}(P_\bullet, \bbC\Gamma),~ d^{n+1}=\hom(\_, d_{n+1})\colon P^n\to P^{n+1}.\] 
The group algebra $\bbC\Gamma$ is a $\bbC$-algebra with involution $\ast$ such that $(a\gamma)^\ast=\bar a\gamma^{-1}$ for $a\in\bbC$ and $\gamma\in\Gamma$.
Via the bases the boundary maps $d_\bullet$ are represented by matrices over~$\bbC\Gamma$. The $\ast$-operation extends to matrices over $\bbC\Gamma$ by taking the transpose and then applying the involution~$\ast$ to the entries. 
Via the matrix representation we obtain the \emph{formal adjoint} $d_n^*\colon P_{n-1}\to P_n$. 
The basis on $P_n$ induces an obvious isomorphism $P_n\xrightarrow{\cong}P^n$. 
The 
resulting diagram 
\begin{equation}\label{eq: adjoint and dual differentials}
\begin{tikzcd}
P_n \ar[d, "d_{n+1}^\ast"']\ar[r, "\cong"] & P^n\ar[d, "d^{n+1}"]\\
P_{n+1}\ar[r, "\cong"] & P^{n+1}
\end{tikzcd}
\end{equation}
commutes. 
    The \emph{(formal) Laplace operator} in degree~$n$ of the chain complex $P_\bullet$ is defined as 
    \begin{equation}\label{eq: laplace operator} \Delta_n\defq d_n^\ast d_n+d_{n+1}d_{n+1}^\ast\colon P_n\to P_n.
    \end{equation}
    It is a $\bbC\Gamma$-linear map. 
    Analogously, we define the Laplace operator $\Delta^n\colon P^n\to P^n$. As a consequence of~\eqref{eq: adjoint and dual differentials} we obtain the following. 
    
\begin{remark}\label{rem: matrix of Laplacian}
The representing square matrices over $\bbC\Gamma$ of $\Delta_n\colon P_n\to P_n$ and $\Delta^n\colon P^n\to P^n$ coincide.
\end{remark}

\begin{remark}\label{rem: chain homotopy Laplacian}
It is easy to see that $\Delta_\bullet\colon P_\bullet\to P_\bullet$ is a chain map. The maps 
$d_{\bullet+1}^\ast\colon P_\bullet\to P_{\bullet+1}$ form a chain homotopy $\Delta_\bullet\simeq 0$ of $\Delta_\bullet$ to the zero map. Similarly for $\Delta^\bullet$. 
\end{remark}

Let $V$ be a $\bbC\Gamma$-module. Note that we can turn every left $\bbC\Gamma$-module into a right $\bbC\Gamma$-module and vice versa through the involution on~$\bbC\Gamma$. 
The isomorphism $P_n\xrightarrow{\cong} P^n$ above induces an isomorphism of $\bbC$-vector spaces 
\[ P_n^V\defq V\otimes_{\bbC\Gamma} P\xrightarrow{\cong} P^n\otimes_{\bbC\Gamma} V\cong \hom_{\bbC\Gamma}\bigl(P_n, V\bigr)\qdef P^n_V.\]
We obtain a similar commutative diagram as in~\eqref{eq: adjoint and dual differentials} for $P_\bullet^V$ and $P^\bullet_V$. The homological and cohomological Laplace operators $\Delta^V_\bullet$ and $\Delta^\bullet_V$ on $P_\bullet^V$ and $P^\bullet_V$ are defined 
similarly with $d_n^\ast$ begin replaced by $(d_n^V)^\ast\defq \id_V\otimes d_n^\ast$. 

\begin{remark}\label{rem: involutive algebra extension}
Let $A$ be a $\bbC$-algebra with involution that is equipped with a $\ast$-homomorphism $\bbC\Gamma\to A$. Then $P_\bullet^A$ is a chain complex of left $A$-modules, and $P^\bullet_A$ is a cochain complex of right $A$-modules. Via the involution on $A$ we can view $P^\bullet_A$ as a cochain complex of left $A$-modules. Then the $\bbC$-isomorphism $P_n^A\xrightarrow{\cong} P^n_A$ becomes an $A$-linear isomorphism, and the Laplace operators are $A$-linear maps. Furthermore, Remark~\ref{rem: matrix of Laplacian} still applies to $\Delta_n^A\colon P_n^A\to P_n^A$ and $\Delta^n_A\colon P^n_A\to P^n_A$. 
\end{remark}

\begin{lemma} \label{lem:lapinv} Let $n\in\bbN$. 
    Let $A$ be a $\bbC$-algebra that is equipped with a  homo\-morphism $\bbC\Gamma\to A$. Let $V$ be a $\bbC\Gamma$-module, and let $W$ be an $A$-module. 
    Then the following statements hold. 
    \begin{enumerate}
        \item If $\Delta_n\colon P_n\to P_n$ is invertible as a $\bbC\Gamma$-linear map, then $H_n(\Gamma, V)=0$ and $H^n(\Gamma, V)=0$.  
        \item If $\Delta_n^V\colon P_n^V\to P_n^V$ is invertible as a $\bbC$-linear map, then $H_n(\Gamma, V)=0$. Similarly for $\Delta^n_V$. 
        \item Let $\Delta_n^A\colon P_n^A\to P_n^A$ be invertible as an $A$-linear map. Then $H_n(\Gamma, W)=0$. If, in addition, $\bbC\Gamma\to A$ is a $\ast$-homomorphism of $\bbC$-algebras with involutions, then also $H^n(\Gamma, W)=0$. 
    \end{enumerate}
\end{lemma}

The point of Lemma~\ref{lem:lapinv} is that the hypothesis is for a single degree~$n$. If all Laplace operator are invertible, then we could deduce it immediately from Remark~\ref{rem: chain homotopy Laplacian}. 

\begin{proof}
All three statements are similar. We only prove the second and write $\Delta_\bullet$ short for $\Delta_\bullet^V$. The reason one obtains the cohomological consequence in the first and third statement is Remark~\ref{rem: matrix of Laplacian} and the last sentence in Remark~\ref{rem: involutive algebra extension}. 
We have 
\[ d_{n+1}d_{n+1}^\ast\Delta_{n}=
d_{n+1}\Delta_{n+1}d_{n+1}^\ast=
\Delta_{n}d_{n+1}d_{n+1}^\ast.\]
Thus 
$\Delta_{n}^{-1}d_{n+1}(d_{n+1})^\ast=d_{n+1}(d_{n+1})^\ast\Delta_{n}^{-1}$.
For an $n$-cycle $z$, we conclude that  
\[ z=\Delta_n^{-1}\bigl(d_n^\ast d_n+d_{n+1}d_{n+1}^\ast\bigr) z=\Delta_{n}^{-1}d_{n+1}d_{n+1}^\ast z=d_{n+1}d_{n+1}^\ast\Delta_{n}^{-1}z. \]
Thus $z$ a boundary.
\end{proof}

\begin{rem} \label{rem:complex}
    Since the bases for $P_\bullet$ are defined over $\bbQ$, so are also the maps $d_\bullet$, $d^\bullet$ and $\Delta_\bullet$, $\Delta^\bullet$.
    Note, in particular, that for a $\bbC\Gamma$-module $V$, taking (co)homology commutes with the complex conjugation.
\end{rem}

\subsection{Banach $\Gamma$-modules} \label{sec:BanM}

Banach spaces are assumed to be complex Banach spaces. 
A \emph{Banach $\Gamma$-module} is a Banach space on which $\Gamma$ acts continuously by isomorphisms.
If the action is by isometries, we speak of an \emph{isometric Banach $\Gamma$-module}.
The $\Gamma$-action on a Banach $\Gamma$-module extends to an action of the group algebra $\bbC \Gamma$.

The cohomology $H^\bullet(\Gamma,V)$ has a canonical topological vector space structure, which is typically non-Hausdorff~\cite{guichardet}. If $\Gamma$ is of type $FP_\infty(\bbQ)$, which we assume (see Convention~\ref{conv:ft}), and $P_\bullet$ is any resolution as in~\eqref{eq: adjoint and dual differentials}, then $\hom_{\bbC\Gamma}(P_\bullet, V)$ is a cochain complex of Banach spaces whose topology induces the one on $H^\bullet(\Gamma, V)$.

The closure $\overline{\{0\}}\subset  H^n(\Gamma,V)$  of $0$ 
is sometimes called the \emph{$n$-th cohomological torsion} of $V$\footnote{Not to be confused with the torsion in the integral cohomology groups!}. 
Quotientening by the cohomological torsion, we obtain the quotient space $\bar{H}^n(\Gamma,V)$. 
This is a Banach space called the \emph{$n$-th reduced cohomology} of $V$. It carries a canonical topology but not a canonical Banach norm. 
Similarly, we define the \emph{$n$-th homological torsion} and the \emph{$n$-th reduced homology}  $\bar{H}_n(\Gamma,V)$.

If $V$ is a Banach $\Gamma$-module, the Banach dual $V^*$ has a natural structure of a Banach $\Gamma$-module as well. 
We refer to the setting in \S\ref{subsec: laplace operator}. 
The cochain complex $P^\bullet_{V^\ast}$ is topologically dual to the chain complex $P_\bullet^V$.  
Similarly, $P_\bullet^{V^\ast}$  is topologically dual to $P^\bullet_V$. 

\begin{lemma}\label{lem: duality reduced cohomology}
    Let $V$ be a Banach $\Gamma$-module. 
    Then
    $H_n(\Gamma, V)$ is Hausdorff if and only if $H^{n+1}(\Gamma, V^\ast)$ is Hausdorff, and $H_n(\Gamma, V^\ast)$ is Hausdorff if and only if $H^{n+1}(\Gamma, V)$ is Hausdorff. 
    If $H^{n}(\Gamma,V^*)$ is Hausdorff then
    \[ H^n(\Gamma, V^\ast) \cong \Bigl(\bar H_n(\Gamma, V)\Bigr)^\ast \]
    and if $H_{n}(\Gamma,V^*)$ is Hausdorff then
    \[ H_n(\Gamma, V^\ast) \cong \Bigl(\bar H^n(\Gamma, V)\Bigr)^\ast. \]
\end{lemma}

\begin{proof}
We write short $C_n=P_n^V$ and $C^n=P^n_{V^\ast}$. For the indexing of (co-)differentials we follow the convention in \S\ref{subsec: laplace operator}.  
By the closed range theorem a continuous linear map between Banach spaces has closed range if and only if its dual map has closed range. Applied to the differential~$d_{n+1}\colon C_{n+1}\to C_n$ this proves that $H_n(\Gamma, V)$ is Hausdorff if and only if $H^{n+1}(\Gamma, V^\ast)$ is Hausdorff. 

Assume that $H^{n}(\Gamma, V^\ast)$ is Hausdorff. Thus $H_{n-1}(\Gamma, V)$ is Hausdorff and 
$\im d_n\subset C_{n-1}$ is closed.  By the Hahn-Banach theorem, 
$C^{n-1} \to (\im d_n)^*$ is surjective.
The map $d^n$ factors as 
\[ C^{n-1} \twoheadrightarrow (\im d_n)^* \cong (C_n/\ker d_n)^* \cong (\ker d_n)^\perp \hookrightarrow C^n. \]
Thus $\im d^n=(\ker d_n)^\perp$ in $C^n$.
Without any assumption on being Hausdorff, we also have 
$\ker d^{n+1}=(\im d_{n+1})^\perp=(\overline{\im d_{n+1}})^\perp$.
Dualizing the short exact sequence
\[
0\to \overline{\im d_{n+1}}\to \ker d_n\to \bar H_n(\Gamma, V)\to 0
\]
yields the short exact sequence 
\[
0\to \Bigl(\bar H_n(\Gamma, V)\Bigr)^\ast\to (\ker d_n)^\ast\to (\overline{\im d_{n+1}})^\ast\to 0.
\]
and we obtain that 
\begin{align*} \Bigl(\bar H_n(\Gamma, V)\Bigr)^\ast&\cong (\ker d_n)^\ast/(\overline{\im d_{n+1}})^\ast \\
&\cong (C^n/(\ker d_n)^\perp)/(C^n/(\overline{\im d_{n+1}})^\perp)\\
&\cong (\overline{\im d_{n+1}})^\perp/(\ker d_n)^\perp\\
&= \ker d^{n+1}/\im d^n \\
&= H^n(\Gamma,V^*).
\end{align*}
The remaining isomorphisms are proved similarly.
\end{proof}

\begin{lemma} \label{lem:torhom}
    Let $V$ be a Banach $\Gamma$-module with $V^\Gamma$ finite dimensional and let $U<V$ be a closed sub-$\Gamma$-module. If $H^1(\Gamma,U)$ is not Hausdorff then also $H^1(\Gamma,V)$ is not Hausdorff.
\end{lemma}

\begin{proof}
If $H^1(\Gamma,V)$ is Hausdorff then the image under $d^1$ of $P^0_V\cong V$ in $P^1_V$ is closed, and isomorphic to $V/V^\Gamma$. Since $V^\Gamma$ is finite dimensional, the image of $U$ is closed in $V/V^\Gamma$, and we deduce that the image under $d^1$ of $P^0_U\cong U$ in $P^1_U$ is closed, thus $H^1(\Gamma,U)$ is Hausdorff.
\end{proof}

\subsection{Ultrapowers}

We fix a non-principal ultrafilter $\omega$ on $\bbN$.
For every Banach space $V$ we consider $\lim_\omega \|\cdot\|$ as a seminorm on the space $\ell^\infty(\bbN,V)$ of bounded sequences. After dividing the kernel of the seminorm, we obtain a norm $\|\cdot\|_\omega$ 
on the quotient space $V_\omega$,  which is easily seen to be complete. We call the resulting Banach space $V_\omega$ the \emph{ultrapower} of $V$.
If $T:U\to V$ is a bounded linear map between Banach spaces, naturally associated map $\ell^\infty(\bbN,U) \to \ell^\infty(\bbN,V)$ descends to a bounded linear map $T_\omega:U_\omega \to V_\omega$. 

Let us denote the category of Banach spaces with bounded linear maps as morphisms by~$\banach$. 

\begin{lemma} \label{lem:openmap}
    Taking ultrapowers yields a functor $\Omega\colon \banach\to\banach$, $\Omega(A)=A_\omega$. Further, if $0\to A\to B\to C\to 0$ is a short exact sequence of Banach spaces, then $0\to A_\omega\to B_\omega\to C_\omega\to 0$ is a short exact sequence. 
\end{lemma}

\begin{proof}
    This is a consequence of the open mapping theorem. 
\end{proof}

\begin{remark}\label{rem: exact category}
An additive functor between abelian categories with the property in Lemma~\ref{lem:openmap} is called \emph{exact}. 
However, the category $\banach$ fails to be abelian since 
a bounded linear map with trivial kernel and trivial cokernel is not an isomorphism in general. Note that the (categorical) cokernel of $T\colon U\to V$ is $\coker(T:U\to V)=V/\overline{\im T}$. The category $\banach$ is still an exact category. See~\cite{buehler} for a concise introduction to exact categories. 

The meaning of an exact functor between exact categories is that it preserves admissible kernel-cokernel pairs, which in the case of $\banach$ reduces to the functor preserving (categorical) cokernels in~$\banach$. 
\end{remark}

In a strong sense and despite Lemma~\ref{lem:openmap}, the ultrapower functor is not exact as a functor of exact categories.   

\begin{lemma} \label{lem:up}
Let $T:U\to V$ be a bounded linear map of Banach spaces with non-closed image.
Then $T_\omega$ is neither injective
nor its image is dense.
Moreover, the inclusions $(\ker T)_\omega\subset \ker (T_\omega)$ and $\overline{\im (T_\omega)}\subset(\overline{\im T})_\omega$ are proper.
\end{lemma}

\begin{proof}
By Lemma~\ref{lem:openmap}, we easily reduce to the case where $T$ is injective and with a dense, non-closed, image. By the open mapping theorem there exists a sequence of unit vectors $u_i \in U$ with $\lim T(u_i)=0$. 
The element represented by $(u_i)$ in $U_\omega$ is in $\ker (T_\omega)$.
Note that also $T^*\colon V^*\to U^*$ is injective and with a dense, non-closed, image.
We thus can find a unit vector $\phi\in (V^*)_\omega$ that is in $\ker ((T^*)_\omega)$.
The image of $\phi$ under the natural map $(V^*)_\omega \to (V_\omega)^*$\footnote{This map is an isometry into, which is surjective iff $V$ super-reflexive \cite[Corollary~7.2]{Heinrich}.} vanishes on $\im (T_\omega)$. Hence also on $\overline{\im (T_\omega)}$.
\end{proof}

Let $F\colon \banach\to \mathsf{Abelian}$ be the forgetful functor to abelian groups. 
The functor $F$ preserves monomorphisms or kernels but does not preserve epimorphisms or cokernels (see Remark~\ref{rem: exact category}). In particular, $F$ is not an exact functor. This is one of the reasons that the homological algebra underlying cohomology with coefficients in Banach spaces is a bit more unpleasant than in a purely algebraic context. The ultrapower functor $\Omega$ is helpful because Lemmas~\ref{lem:openmap} and~\ref{lem:up} imply the following corollary. It says that we can express the preservation of epimorphisms and cokernels by~$F$  within~$\banach$, thus staying in the functional-analytic context. 

\begin{cor}\label{cor: categorical epi and cokernel}
    Let $f\colon A\to B$ be a (categorical) epimorphism in~$\banach$\footnote{This means that $f$ has dense image.}. Then $F(f)$ is an epimorphism   in~$\mathsf{Abelian}$ if and only if $\Omega(f)$ is an epimorphism in $\banach$. Similary, $F(\coker f)$ is the cokernel of $F(f)$ if and only if $\Omega(\coker f)$ is the cokernel of $\Omega(f)$. 
\end{cor}

The ultrapower functor $\Omega$ restricts to a functor on Banach $\Gamma$-modules and extends to a functor on chain complexes in~$\mathsf{Ban}$. Let $P_\bullet$ be a resolution as in~\eqref{eq: projective resolution}. Since $\Omega$ commutes with direct finite sums, we obtain that  
\begin{equation}\label{eq: ultra identification} \bigl(P_\bullet^V\bigr)_\omega=\bigl(V\otimes_{\bbC\Gamma} P_\bullet\bigr)_\omega\cong V_\omega\otimes_{\bbC\Gamma}P_\bullet=P_\ast^{V_\omega}.
\end{equation}
Similarly, for the cochain complex.

\begin{lemma} \label{lem:uptrick}
    Let $n\in\bbN$. Let $\mathcal{V}$ be a class of Banach $\Gamma$-modules that is closed under taking ultrapowers. 
    \begin{enumerate}
    \item If $\bar H_n(\Gamma,V)=0$ for every $V\in \mathcal{V}$, then  $H_n(\Gamma,V)=0$ and $H_{n-1}(\Gamma,V)$ is Hausdorff for every $V\in \mathcal{V}$. 
    \item 
    If $\bar H^n(\Gamma,V)=0$ for every $V\in \mathcal{V}$, then  $H^n(\Gamma,V)=0$ and $H^{n+1}(\Gamma,V)$ is Hausdorff for every $V\in \mathcal{V}$. 
    \end{enumerate}
\end{lemma}

\begin{proof}
We only prove the first statement as the one for the second statement is similar. 
Let $V\in \mathcal{V}$.
If $H_{n}(\Gamma,V)\neq 0$, then it is not Hausdorff. Thus $d_{n+1}$ has a non-closed image. 
By Lemma~\ref{lem:up}, $\overline{\im (d_{n+1})_\omega}\subset (\overline{\im d_{n+1}})_\omega$ is a proper inclusion.  
In particular, $\overline{\im (d_{n+1})_\omega}\subset\ker ((d_{n})_\omega)$ is proper. In view of~\eqref{eq: ultra identification}, this contradicts $\bar H_n(\Gamma,V_\omega) \neq 0$.
Hence $H_{n}(\Gamma,V)= 0$.

If $H_{n-1}(\Gamma,V)$ is not Hausdorff, then $d_{n}$ has a non-closed image, and by Lemma~\ref{lem:up}, $(\ker d_n)_\omega\subset\ker ((d_n)_\omega)$ is proper. 
In particular, $\im ((d_{n+1})_\omega)\subset\ker ((d_{n})_\omega)$ is proper, contradicting $H_{n}(\Gamma,V_\omega)= 0$. 
Hence $H_{n-1}(\Gamma,V)$ is Hausdorff.
\end{proof}

For a Banach $\Gamma$-module $V$, an almost invariant sequence is a sequence of vectors $v_n\in V$ such that for every $g\in \Gamma$, $\lim_n (1-g)v_n=0$.
We end this subsection with the following observation, see ~\cite{badsau}*{Section~3}.

\begin{lemma}\label{lem: fixed points and ultrapower}
Let $V$ be a Banach $\Gamma$-module. If $V$ has an almost invariant sequence of unit vectors, then $V_\omega$ has non-zero $\Gamma$-invariant vectors.
\end{lemma}

\begin{proof}
If $(v_n)$ is an almost invariant sequence of unit vectors in $V$, then the element in $V_\omega$ represented by $(v_n)$ is a non-trivial $\Gamma$-fixed vector.
\end{proof}

\subsection{Hilbert $C^\ast$-modules}

Let $A$ be a $C^\ast$-algebra. 
A \emph{Hilbert $A$-module} is an $A$-module $E$ with an $A$-valued inner product $\langle\_, \_\rangle\colon E\times E\to A$ such that $E$ is complete with respect to the ($\bbC$-valued) norm $\norm{x}=\norm{\langle x,x\rangle}_A^{1/2}$. We refer to the book~\cite{lance} as a background reference. 

The category of Hilbert $A$-modules with bounded $A$-linear maps as morphisms for an arbitrary $C^\ast$-algebra~$A$ is very similar to the special case~$A=\bbC$ which is just about Hilbert spaces and bounded operators. But there are some important differences. A closed $A$-submodule $F\subset E$ is \emph{complemented} if $E=F\oplus F^\perp$ where $F^\perp$ is defined in the obvious sense. Not every closed $A$-submodule is complemented. Further, 
not every bounded $A$-linear map has an adjoint. We usually work in the category of Hilbert $A$-modules with bounded and adjointable $A$-linear maps as morphisms. If $V, W$ are Hilbert $A$-modules, we denote the set of bounded adjointable maps from $V$ to~$W$ by $B(V,W)$. Furthermore, $B(V,V)$ is a $C^\ast$-algebra~\cite{lance}*{p.~8}. So a bounded adjointable  operator $f\colon V\to V$ is \emph{positive} if it is positive as an element of $B(V,V)$. This is implied by the positivity $\langle f(x), x\rangle\ge 0$ of the $A$-valued inner product by~\cite{lance}*{Lemma~4.1}. 

Let $E$ be a finitely generated projective $A$-module in the algebraic sense\footnote{In the case of a commutative $C^\ast$-algebra $C(X)$ such an $E$ corresponds to a finite-dimensional vector bundle over~$X$.}. 
Then $E$ embeds as a direct summand of some finitely generated free module $A^n$ and inherits the structure of a Hilbert $A$-module from that embedding. The Hilbert $A$-structure is independent of the choice up to unitary isomorphism. This follows from combining~\cite{lance}*{Corollary~3.3 and Proposition~3.8}. 
Furthermore, every $A$-linear map between finitely generated projective $A$-modules is bounded and adjointable. In case of finitely generated free $A$-modules the $A$-linear map~$f$ is represented by a matrix $M$, and the adjoint of~$f$ is represented by the matrix obtained from $M$ by taking the transpose and the element-wise adjoint. In the general case one uses~\cite{lance}*{Corollary~3.3} again. 

%In particular, if $P_\bullet$ is a chain complex of finitely generated projective $\bbC\Gamma$-modules, then $C^\ast\Gamma\otimes_{\bbC\Gamma} P_\bullet$ is a complex of Hilbert $C^\ast\Gamma$-modules with bounded and adjointable differentials. 
%Similarly for the $C^\ast$-algebra $C_0^\ast\Gamma$ provided $\Gamma$ has property~T. 

\begin{lemma} \label{lem:Acohom}
%Let $P_\bullet\to\bbC$ be a projective $\bbC\Gamma$-resolution of finite type. 
Let $A$ be a unital C*-algebra endowed with a $\ast$-homomorphism $\mathbb{C}\Gamma\to A$.
Let $\Delta_\bullet^A$ denote the Laplace operators as in~\eqref{eq: laplace operator} and Remark~\ref{rem: involutive algebra extension}. Let $n\in\bbN$. 
The following conditions are equivalent. 
\begin{enumerate}
    \item $H^n(\Gamma, A)=0$ and $H^{n+
    1}(\Gamma, A)$ is Hausdorff. 
    \item $H_n(\Gamma, A)=0$ and $H_{n-1}(\Gamma, A)$ is Hausdorff. 
    \item $\Delta_n^A$ is invertible as an $A$-linear map or, equivalently, as a bounded operator of Hilbert $A$-modules.  
\end{enumerate}
\end{lemma}

\begin{proof}
Ad $(3)\Rightarrow (2)$. 
We refer to the setting in \S\ref{subsec: laplace operator} but drop the superscript $A$ in the notation of the differentials. 
The vanishing of $H_n(\Gamma, A)$ 
follows from Lemma~\ref{lem:lapinv}. 
By~\cite{lance}*{Proposition~3.2} and its proof the subspaces $\ker(d_n)$ and $\ker(d_{n+1}^\ast)$ of $P_n^A$ are complementable. Further, $\ker(d_n)^\perp$ is the closure of $\im(d_n^\ast)$, and $\ker(d_{n+1}^\ast)^\perp$ is the closure of $\im(d_{n+1})$. Using that one shows the orthogonal decomposition 
\[\ker(d_n)^\perp\oplus \ker(d_{n+1}^\ast)^\perp=\ker(\Delta_n)^\perp=\{0\}^\perp.
\]
The operators $d_n^\ast d_n$ and $d_{n+1}d_{n+1}^\ast$ restrict to operators of the left and right summand, respectively. Hence \begin{equation}\label{eq: direct sum Laplacian} \Delta_n=d_n^\ast d_n\vert_{\ker(d_n)^\perp} \oplus d_{n+1}d_{n+1}^\ast\vert_{\ker(d_{n+1}^\ast)}.
\end{equation}
In particular, $d_n^\ast d_n\colon \ker(d_n)^\perp\to \ker(d_n)^\perp$ is invertible too. This implies that $d_n$ has a closed image. 

Ad $(2)\Rightarrow (3)$. 
First we show that $\Delta_n$ is injective. Since $\im(d_{n+1})=\ker(d_n)$ is closed, $\im(d_{n+1})$ is complementable by~\cite{lance}*{Proposition~3.2}. Further, $\ker(d_{n+1}^\ast)$ is the orthogonal complement of $\im(d_{n+1})$~\cite{lance}*{p.~23}. Therefore $\ker(d_n)\cap \ker(d_{n+1}^\ast)$ injects into $H_(\Gamma, A)$. Thus $\ker(d_n)\cap \ker(d_{n+1}^\ast)=0$. On the other hand, we have $\langle \Delta_n x, x\rangle=\norm{d_n x}^2+\norm{d_{n+1}^\ast x}^2$. This implies that $\Delta_n$ is a positive operator~\cite{lance}*{Lemma~4.1} from which we can conclude that $\Delta_n$ has a square root. Hence $\langle \Delta_n x, x\rangle=0$
 implies that $x=0$ and we obtain that \[\ker(\Delta_n)=\ker(d_n)\cap \ker(d_{n+1}^\ast)=\{0\}.\]
 As above we conclude from the injectivity of~$\Delta_n$ the direct sum decomposition~\eqref{eq: direct sum Laplacian}. Because of  $H_n(\Gamma, A)=0$ the differential $d_{n+1}$ has closed image which implies that the right summand in~\eqref{eq: direct sum Laplacian} is invertible. By assumption $d_n$ has a closed image which implies that also the left summand in~\eqref{eq: direct sum Laplacian} is invertible. Hence $\Delta_n$ is invertible. 

 Ad $(1)\Leftrightarrow (3)$.  The invertibility of the homological Laplace operator is equivalent to the invertibility of the cohomological Laplace operator. See Remark~\ref{rem: involutive algebra extension}. The proof of the equivalence $(1)\Leftrightarrow (3)$ is completely analogous to the one for $(2)\Leftrightarrow (3)$.
 \end{proof}   

\subsection{Unitary $\Gamma$-modules}

We use \emph{unitary $\Gamma$-module} as a synonym for \emph{unitary representation of $\Gamma$}.
Every unitary $\Gamma$-module extends naturally to a $*$-representation of the group algebra $\mathbb{C}\Gamma$ and vice versa.
In the following the Laplace operator is similarly defined as in \S\ref{subsec: laplace operator} with respect to a free  resolution $P_\bullet$ of finite type. 
The (homological) Laplace operator becomes a bounded operator $\Delta_n^V\colon P_n^V\to P_n^V$ between Hilbert spaces. Note that $P_n^V$ is a finite sum of copies of~$V$. 

\begin{lemma} \label{lem:Laplacian}
    Let $n\in\bbN$. Let $V$ be a unitary $\Gamma$-module. 
    Then  
    \[ \bar H^n(\Gamma,V)\cong \bar H_n(\Gamma,V)\cong \ker \Delta_n^V, \]
    and the following conditions are equivalent.
    \begin{enumerate}
        \item $H^n(\Gamma,V)=0$ and $H^{n+1}(\Gamma,V)$ is Hausdorff.
        \item $H_n(\Gamma,V)=0$ and $H_{n-1}(\Gamma,V)$ is Hausdorff.
        \item $\Delta_n^V$ is  invertible as a bounded operator. 
    \end{enumerate}
\end{lemma}

\begin{proof}
The first statement is well known. See, for example, \cite[Lemma 1.18]{Lueck} or \cite[Proposition 16.1]{Bader-Nowak}. The isomorphism $\bar H_n(\Gamma, V)\cong \ker\Delta_n^V$  follows from the orthogonal decompositions
\begin{align*} P_n^V&=\ker(d_n)\cap\im(d_{n+1})^\perp\oplus \overline{\im(d_{n+1})}\oplus\ker (d_n)^\perp\\
    &=\ker(d_n)\cap\ker(d_{n+1})\oplus \overline{\im(d_{n+1})}\oplus \ker (d_n)^\perp\\
    &=\ker(\Delta_n^V)\oplus \overline{\im(d_{n+1})}\oplus \ker (d_n)^\perp.
\end{align*}
Similarly for the cohomological version.    
The equivalence (1) $\Leftrightarrow$ (2) follows by Lemma~\ref{lem: duality reduced cohomology} and the fact that $V\cong V^\ast$ as $\bbC\Gamma$-modules. Note that if $V$ is a right $\bbC\Gamma$-module, then $V^\ast$ is a left $\bbC\Gamma$-module and vice versa. So in the isomorphism we implicitly use the involution on $\bbC\Gamma$ to regard $V$ as a left $\bbC\Gamma$-module. 
Lemma~\ref{lem:lapinv}(3) gives (3) $\Rightarrow$ (1) and (2), see Remark~\ref{rem: involutive algebra extension}.
On the other hand (2) implies that both maps $d_n$ and $d_{n+1}$ have closed ranges, so also their adjoints $d_n^*$ and $d_{n+1}^*$ have closed ranges, and so do the compositions $d_n^*d_n$ and $d_{n+1}d_{n+1}^*$.
Since the images of these maps are orthogonal, we get that also $\Delta_n$ has a closed image. Since $\Delta_n$ is a positive operator injective operator, we conclude that it is invertible, thus (2) $\Rightarrow$ (3).
Similarly, (1) $\Rightarrow$ (3).
See also see \cite[Proposition 16.5]{Bader-Nowak} for (1) $\Leftrightarrow$ (3).
\end{proof}

The ultrapower of a Hilbert space is a Hilbert space, and accordingly, the collection of unitary $\Gamma$-module is closed under taking ultrapowers. 
More than that, we have the following.

\begin{lemma} \label{lem:Aup}
    Let $A$ be a unital C*-algebra with a $\ast$-homomorphism  $\mathbb{C}\Gamma\to A$.
    Then every Hilbert $*$-representation of $A$ is a unitary $\Gamma$-module
and this collection of unitary $\Gamma$-modules is closed under ultrapowers.
\end{lemma}

\begin{proof}
    This follows from the functoriality of the ultrapower.
\end{proof}

\begin{rem}
No continuity is a priori assumed for Hilbert $*$-representations of $A$, but it follows automatically from the $*$-homomorphism property. 
\end{rem}

\begin{lemma} \label{lem:basic}
    Let $A$ be a unital C*-algebra with a $\ast$-homomorphism  $\mathbb{C}\Gamma\to A$. 
    For every $n\in\bbN$, the following conditions are equivalent.
    \begin{enumerate}
        \item The Laplace operator $\Delta_n^A$ is invertible as a $A$-linear map.
        \item For every Hilbert $*$-representation $V$ of $A$, $\bar H^n(\Gamma,V)=0$.
        \item For every Hilbert $*$-representation $V$ of $A$, $H^n(\Gamma,V)=0$.
        \item For every Hilbert $*$-representation $V$ of $A$, $H^n(\Gamma,V)=0$ and $H^{n+1}(\Gamma, V)$ is Hausdorff.
        \item For every $A$-module $V$, $H^n(\Gamma,V)=0$.
        \item $H^n(\Gamma, A)=0$ and $H^{n+1}(\Gamma, A)$ is Hausdorff. 
    \end{enumerate}
    Further, they all imply the following.
    \begin{enumerate} \setcounter{enumi}{6}
        \item For every Banach $A$-module $V$, $H^{n+1}(\Gamma,V)$ is Hausdorff.
    \end{enumerate}
    Moreover, the conditions (1)-(6) are equivalent to (2') - (6') and they imply (7'), which are the analogous conditions where cohomology is replaced by homology and in (4'),(6') and (7') we also replace $n+1$ by $n-1$.
\end{lemma}

\begin{proof}
We will only consider the cohomological versions, as the proof for homologies are similar.
By Lemma~\ref{lem:Acohom}, (1) $\Leftrightarrow$ (6)
and by Lemma~\ref{lem:Laplacian}, (1) $\Leftrightarrow$ (4).
By Lemma~\ref{lem:lapinv}, (1) $\Rightarrow$ (5).
By Lemma~\ref{lem: duality reduced cohomology}, (5) $\Rightarrow$ (7).
We conclude that (5) $\Rightarrow$ (4), thus (5) $\Leftrightarrow$ (1).
At this point we have 
(1) $\Leftrightarrow$ (4) $\Leftrightarrow$
(5) $\Leftrightarrow$
(6) $\Rightarrow$ (7)
and clearly (4) $\Rightarrow$ (3) $\Rightarrow$ (2).
By Lemma~\ref{lem:Aup} and Lemma~\ref{lem:uptrick}, (2) $\Rightarrow$ (4), thus (2) $\Leftrightarrow$ (3) $\Leftrightarrow$
(4). This finishes the proof.
\end{proof}

\subsection{Universal envelopes}

The forgetful functor from the category of von Neumann algebras to the category of $C^\ast$-algebras has a left adjoint called the \emph{von Neumann envelope functor}. By a result 
of Sherman-Takeda the von Neumann envelope of a $C^\ast$-algebra is its double 
dual (as a von Neumann algebra). See~\cite{takesaki}*{III.2}.

Every unitary $\Gamma$-module extends naturally to a Hilbert $*$-representation of the group algebra $\mathbb{C}\Gamma$ and vice versa.
Every Hilbert $*$-representation of the group algebra $\mathbb{C}\Gamma$ extends naturally to a $*$-representation of the maximal $C^\ast$-algebra $C^*\Gamma$ and vice versa. 
Every Hilbert $*$-representation of $C^*\Gamma$ extends further, by the Sherman–Takeda theorem, to a Hilbert normal $*$-representation of its von Neumann envelope, $W^*\Gamma=(C^*\Gamma)^{**}$ and vice versa.
We thus get four equivalent versions of the category of unitary $\Gamma$-modules.

\begin{rem} \label{rem:intermed}
    For every intermediate C*-algebra $C^*\Gamma\subset A\subset W^*\Gamma$,
    every unitary $\Gamma$-representations extends to a Hilbert $*$-module of $A$ -- but possibly in a non-unique way. 
    For example, a $*$-representation of $W^*\Gamma$ that is not normal restricts to a unitary $\Gamma$-representations, which in turn extends to a normal $*$-representation of $W^*\Gamma$.   
\end{rem}

By the double commutant theorem applied to the universal representation of $C^\ast\Gamma$  the von Neumann algebra $W^*\Gamma$ has an
entirely algebraic description as the endomorphism algebra of the forgetful functor from unitary $\Gamma$-modules to Hilbert spaces. 
Its center $Z^*\Gamma$ is thus identified as the endomorphisms of identity functor, that is the \emph{Bernstein center}, of the category of unitary $\Gamma$-modules.

\subsection{The augmentation ideals and the Kazhdan projection}
\label{subsec: kazhdan proj}
The trivial representation of $\Gamma$ extends to the \emph{augmentation homomorphism} $\epsilon:\mathbb{C}\Gamma\to \mathbb{C}$.
By the corresponding universal properties, $\epsilon$ extends to a C*-algebra homomorphism $C^*\Gamma\to \mathbb{C}$ and a von Neumann algebra homomorphism $W^*\Gamma\to \mathbb{C}$ denoted by the same symbol~$\epsilon$.
The corresponding kernels are denoted $C_0^*\Gamma\lhd C^*\Gamma$ and $W_0^*\Gamma\lhd W^*\Gamma$.

We identify all spaces above as subspaces of $W^*\Gamma$ and view them as $\Gamma$-modules through the action by left multiplication.

\begin{lemma} \label{lem:Winv}
    The subspace of $\Gamma$-invariants in the $\Gamma$-module $W_0^*\Gamma$ is zero. 
\end{lemma}

\begin{proof}
Let $a\in (W_0^*\Gamma)^\Gamma$. Consider a faithful normal $*$-representation of $W_0^*\Gamma$, which we consider as a $\Gamma$-unitary representation with no non-trivial invariants. The image of the operator $a$ is in the $\Gamma$-invariants, hence $a=0$.
\end{proof}

The augmentation $\epsilon:\mathbb{C}\Gamma\to \mathbb{C}$ splits if and only if $\Gamma$ is finite, sending $1\in \mathbb{C}$ to the idempotent in $\mathbb{C}\Gamma$ defined by $\gamma \mapsto |\Gamma|^{-1}\cdot \gamma$ for $\gamma\in \Gamma$. 
Its extension to $W^*\Gamma$ splits for every group $\Gamma$.
The orthogonal projection on the $\Gamma$-invariants in each unitary representation forms an element in the Bernstein center that is called the \emph{Kazhdan projection} $z\in Z^*\Gamma<W^*\Gamma$. 
The Kazhdan projection is a central projection, 
and the map sending 1 to $z$ splits $\epsilon$.
Its image $W^*\Gamma z=\mathbb{C}z$ is a one dimensional ideal and $W_0^*\Gamma=W^*\Gamma(1-z)$ is its direct complement.
In particular, $W_0^*\Gamma$ is a von Neumann algebra itself endowed with a group homomorphism from $\Gamma$ to its unitary operators.
Every unitary representation of $\Gamma$ with no non-trivial invariants extends to a normal $*$-representation of $W_0^*\Gamma$ and vice versa.

Lemma~\ref{lem:Winv} implies that $(W^*\Gamma)^\Gamma=\mathbb{C}z$.
In view of $C_0^*\Gamma < W_0^*\Gamma$ we obtain that $(C_0^*\Gamma)^\Gamma=0$,
while $(C^*\Gamma)^\Gamma$ is either $0$ or $\mathbb{C}z$ depending on whether $z\in C^*\Gamma$. 

\begin{lemma}[{\cite{Akemann-Walter}*{Lemma 2}}] \label{lem:kproj}
A group $\Gamma$ has property T if and only if the Kazhdan projection $z$ is in $C^*\Gamma$.
\end{lemma}

Equivalently, $\Gamma$ has property T if and only if $\epsilon\colon C^*\Gamma\to \bbC$ splits.
This lemma is well known. As the proof is short, and for its importance in our discussion, we will prove it below.
Before, we make the following convention.

\begin{convention} \label{conv:kproj}
        Under the assumption $z\in C^*\Gamma$ we identify $C_0^*\Gamma\cong C^*\Gamma/\mathbb{C}z$, giving it the structure of a unital C*-algebra whose double dual is the von Neumann algebra $W_0^*\Gamma$. 
        Every unitary representation of $\Gamma$ which has no non-trivial invariants extends to a $*$-representation of $C_0^*\Gamma$ and vice versa.       
\end{convention}

\begin{proof}
If $\Gamma$ has property T, then the trivial representation is isolated in its unitary dual.  Hence it corresponds to a central projection in $C^*\Gamma$, 
which must be $z$, the unique invariant central projection in $W^*\Gamma$.
If $z\in C^*\Gamma$, then the class of unitary representations of $\Gamma$ that have no non-trivial invariants coincides with the class of $*$-representation of the unital algebra $C_0^*\Gamma$, which is closed under taking ultrapowers. By Lemma~\ref{lem:Aup} and
Lemma~\ref{lem: fixed points and ultrapower} such representations have no almost invariant vectors. It follows that $\Gamma$ has property~T.
\end{proof}

We conclude that $H^0(\Gamma,C^*\Gamma)$ detects property T.
We note that also $H_0(\Gamma,C^*\Gamma)$ and 
    $H^1(\Gamma,C^*\Gamma)$ detect property~T.
    Indeed, we will see in Lemma~\ref{lem:yesT} that if $\Gamma$ has property T then $H_0(\Gamma,C^*\Gamma)\cong \mathbb{C}$ and 
    $H^1(\Gamma,C^*\Gamma)=0$, contrasting the following lemma.

\begin{lemma} \label{lem:nonT}
    If $\Gamma$ does not have property~T, then  
    $H_0(\Gamma,C^*\Gamma)$ and 
    $H^1(\Gamma,C^*\Gamma)$ are non-Hausdorff.  
\end{lemma}

\begin{rem} \label{rem:nonT}
If $\Gamma$ does not have property~T then Lemma~\ref{lem: duality reduced cohomology} implies that $H_0(\Gamma,W^*\Gamma)$ is non-Hausdorff and using Lemma~\ref{lem:torhom} we also get that $H^1(\Gamma,W^*\Gamma)$ is non-Hausdorff.
Using the long exact sequences associated with the maps $\epsilon$, one sees that $H_0$ and $H^1$ are non-Hausdorff also for $C_0^*\Gamma$ and $W_0^*\Gamma$ coefficients.
\end{rem}

\begin{proof}
If $H^1(\Gamma,C^*\Gamma)$ is Hausdorff, then Lemma~\ref{lem: duality reduced cohomology} implies that $H_0(\Gamma,(C^*\Gamma)^{*})$ is Hausdorff and isomorphic to 
$H^0(\Gamma,C^*\Gamma)^*=0$. By Lemma~\ref{lem: duality reduced cohomology}, $H^0(\Gamma,(C^*\Gamma)^{**})=H_0(\Gamma,(C^*\Gamma)^{*})^*=0$, which contradicts $z\in (W^*\Gamma)^\Gamma$.
This shows that $H^1(\Gamma,C^*\Gamma)$ is not Hausdorff.

Using Lemma~\ref{lem: duality reduced cohomology} we will show that $H_0(\Gamma,C^*\Gamma)$ is not Hausdorff by showing that $H^1(\Gamma,(C^*\Gamma)^*)$ is not Hausdorff.
The restriction to $\Gamma\subset C^*\Gamma$ gives an embedding $(C^*\Gamma)^* \hookrightarrow \ell^\infty(\Gamma)$ whose image is the Fourier-Stieltjes algebra of all matrix coefficients. By Eymard's theorem~\cite{eymard}, the state space of $C^*\Gamma$ is mapped onto the space $\text{PD}_1(\Gamma)$ of positive definite functions of norm~$1$ under the embedding. Further, the restriction of the weak*-topology from both spaces coincide on $\text{PD}_1(\Gamma)$.

Let $V$ be a unitary $\Gamma$-module $V$ with $V^\Gamma=0$ but with an almost invariant sequence of unit vectors $v_n$. 
The matrix coefficients $\phi_n$ of~$v_n$, given by $\phi_n(\gamma)=\langle \gamma v_n, v_n\rangle$, are a sequence of almost invariant unit vectors in $(C^*\Gamma)^*$. 
Suppose that the $\phi_n$ converge in norm to an invariant function, in this case the function constant~$1$. Then the real part satisfies $\Re\langle \gamma v_n,v_n\rangle > 1/2$ for all $\gamma\in \Gamma$ and $n\gg 1$. The $\Gamma$-action on the closed convex hull of the $\Gamma$-orbit of $v_n$ thus has a fixed point by the lemma of the centre. The previous estimate shows that it is non-zero, which contradicts $V^\Gamma=0$. So the sequence $\phi_n$ does not converge in norm to an invariant function. 
This implies that $H^1(\Gamma,(C^*\Gamma)^*)$ is not Hausdorff.
\end{proof}

\begin{rem}
    The injection $(C^*\Gamma)^* \hookrightarrow \ell^\infty(\Gamma)$ alluded to above shows that $((C^*\Gamma)^*)^\Gamma$ is one dimensional, spanned by $\epsilon$.
\end{rem}

We conclude this section with the following observation.

\begin{lemma} \label{lem:comp}
    Assume $\Gamma$ has property T.
    Then every $C^*\Gamma$ module $V$, the invariants submodule $V^\Gamma$ has a direct complement on which the $C^*\Gamma$ action factors via a unital action of $C^*_0\Gamma$.
    In particular, if $V^\Gamma=0$ then $V$ is a module for the unital algebra $C_0^*\Gamma$.
\end{lemma}

\begin{proof}
    We observe that $V^\Gamma=zV$ and consider its complement $(1-z)V$.
\end{proof}

\subsection{C*-algebraic reformulation of higher T}\label{subsec: C-algebraic reformulation}

Higher property~T (see~Definition~\ref{def:higherT}) is defined through the vanishing of group cohomology for coefficients that are unitary representations.
The following theorem (stated in the introduction as Theorem~\ref{thm:criteria}) is a reformulation using different types of coefficient modules.

\begin{theorem_o}[stated in the introduction as Theorem~\ref{thm:criteria}]
For $n>0$, the following are equivalent. 
\begin{enumerate}
    \item The group~$\Gamma$ has property $[T_n]$.
    \item $H^k(\Gamma,C^*\Gamma)=0$ for all $k\in\{1,\dots, n\}$ and $H^{n+1}(\Gamma, C^\ast\Gamma)$ is Hausdorff. 
    \item $H_k(\Gamma,C^*\Gamma)=0$ for all $k\in\{1,\dots, n\}$ and $H_0(\Gamma, C^\ast\Gamma)$ is Hausdorff. 
\end{enumerate}
A similar equivalence holds for $(T_n)$ with $C^\ast\Gamma$ replaced by $C_0^\ast\Gamma$.
\end{theorem_o}

\begin{rem} \label{rem:criteria}
    In the above theorem, the role of $C^*\Gamma$ 
    could be equivalently replaced by any intermediate C*-algebra $C^*\Gamma\subset A\subset W^*\Gamma$, and in particular by $W^*\Gamma$ itself.
    Similarly, the role of $C_0^*\Gamma$ 
    could be equivalently replaced by any intermediate C*-algebra $C_0^*\Gamma\subset A\subset W_0^*\Gamma$, and in particular by $W_0^*\Gamma$ itself.
\end{rem}

\begin{proof}
    For $[T_n]$ the theorem follows at once from Lemma~\ref{lem:basic} upon setting $A=C^*\Gamma$.
    We now consider the proof for $(T_n)$.
    Assume first that $\Gamma$ does not have $T$.
    Then, for every $n>0$, none of the conditions (1)-(3) is satisfied, by Lemma~\ref{lem:nonT} and Remark~\ref{rem:nonT}.
    We thus assume as we may that $\Gamma$ has T, and in view of Lemma~\ref{lem:kproj}, we apply Convention~\ref{conv:kproj}, thus considering $C_0^*\Gamma$ as a unital C*-algebra extending $\bbC\Gamma$,
    such that every unitary representation of $\Gamma$ which has no non-trivial invariants extends to a $*$-representation of $C_0^*\Gamma$ and vice versa.     
    The theorem now follows again from Lemma~\ref{lem:basic} upon setting $A=C_0^*\Gamma$.

    Next we explain Remark~\ref{rem:criteria} for an intermediate C*-algebra $C^*\Gamma\subset A\subset W^*\Gamma$.
    If $\Gamma$ does not have property~T, (2) and (3) are not satisfied~by Lemma~\ref{lem:torhom}, so we may assume that $\Gamma$ has property~T.
    The proof follows from Remark~\ref{rem:intermed} and its obvious $C^*_0\Gamma$ analogue.
\end{proof}

\begin{theorem_o}[stated in the introduction as Theorem~\ref{thm:TnCor}]
    If $\Gamma$ has property $[T_n]$, 
    then 
for every $C^*\Gamma$-module $V$ and for every $1\leq k \leq n$, $H^n(\Gamma,V)=H_k(\Gamma,V)=0$ and if $V$ is a Banach $C^*\Gamma$-module then $H^{n+1}(\Gamma,V)$ is Hausdorff.

If $\Gamma$ has property $(T_n)$, then the same holds for every $0\leq k \leq n$, upon replacing $C^*\Gamma$ by $C_0^*\Gamma$. 
In particular, if $\Gamma$ has property $(T_n)$, then for every unitary $\Gamma$-module $V$, $H^{n+1}(\Gamma,V)$ is Hausdorff.
\end{theorem_o}

\begin{proof}
The only part which does not follow immediately from Lemma~\ref{lem:basic}
is the last statement.
However, using the fact that $H^{n+1}(\Gamma,\bbC)$ is finite dimensional by Convention~\ref{conv:ft}, we get
\[ H^{n+1}(\Gamma,V)= H^{n+1}(\Gamma,(V^\Gamma)^\perp) \oplus H^{n+1}(\Gamma,\bbC)\otimes V^\Gamma.\]
This is Hausdorff provided $\Gamma$ has property $(T_n)$ and so the first summand vanishes.
\end{proof}

The following lemma contrasts Lemma~\ref{lem:nonT}.

\begin{lemma} \label{lem:yesT} 
If $\Gamma$ has property T then $H_0(\Gamma,C_0^*\Gamma)=0$,     $H_1(\Gamma,C_0^*\Gamma)=0$, $H^1(\Gamma,C_0^*\Gamma)=0$, $H_0(\Gamma,C^*\Gamma)\cong \mathbb{C}$, $H_1(\Gamma,C^*\Gamma)=0$ and 
    $H^1(\Gamma,C^*\Gamma)=0$.
\end{lemma}

\begin{proof}
  All equations but $H_0(\Gamma,C^*\Gamma)\cong \mathbb{C}$ follow immediately from Theorem~\ref{thm:TnCor}, and the latter follows from 
  $H_0(\Gamma,C_0^*\Gamma)=0$ given $C^*\Gamma=C_0^*\Gamma \oplus \bbC z$.
\end{proof}

\subsection{von Neumann algebras}\label{subsec: vN algebras}

We now turn to consider von Neumann algebras.
Given a group $\Gamma$, a von Neumann algebra $M$ and a group homomorphism $\Gamma \to \mathcal{U}(M)$ to the group of unitary elements of~$M$, we consider $M$ as a $\bbC\Gamma$-module with respect to the left action.
We call such modules \emph{$\Gamma$-unitary von Neumann algebras}. 
Every von Neumann algebra $M$ has a unique predual, which we denote $M_*$, and a one-parameter collection of $M$-bimodules, interpolating between $M_*$ and $M$ which we now consider.
For $1\leq p \leq \infty$, there is an associated \emph{non-commutative $L^p$-space} $L^p(M)$, where $L^\infty(M)\cong M$ and $L^1(M)\cong M_*$. For $1<p<\infty$, $L^p(M)$ is super-reflexive and $L^p(M)^* \cong L^q(M)$ where $1/p+1/q=1$.
In particular, $L^2(M)$ is a Hilbert space.
For more information, see\cite{Pisier-Xu}.
%The following theorem is a corollary of Theorem~\ref{thm:TnCor}.
%For the definition of the notion $FP_\infty(\bbQ)$, see \S\ref{subsec:FP}.

\begin{theorem} \label{thm:vNcriterion}
 %   Let $\Gamma$ be a group of type~$FP_\infty(\bbQ)$
  Let $M$ be a $\Gamma$-unitary von Neumann algebra.
   If $\Gamma$ has property $[T_n]$, then 
for every $0\le k\leq n$ and $1\leq p\leq \infty$, 
$H^k(\Gamma,L^p(M))=0$ and $H_k(\Gamma,L^p(M))=0$, 
and in particular $H^k(\Gamma, M)=0$ and $H_k(\Gamma,M)=0$.

If the group~$\Gamma$ has property $(T_n)$, then 
the inclusion $L^p(M)^\Gamma\subset L^p(\Gamma)$ induces isomorphisms in homology and cohomology for for every $0\le k\leq n$ and $1\leq p\leq \infty$.
Further, $L^p(M)^\Gamma=0$ for every $1\leq p\leq \infty$ if and only if $L^p(M)^\Gamma=0$ for some $1\leq p\leq \infty$, and in this case
$H^k(\Gamma,L^p(M))=0$ and $H_k(\Gamma,L^p(M))=0$,
for every $0\le k\leq n$ and $1\leq p\leq \infty$.
In particular, this is the case provided either $M^\Gamma=0$ or $(M_*)^\Gamma=0$.
\end{theorem}

\begin{proof}
We note that the group homomorphism $\Gamma \to \mathcal{U}(M)$ extends to a C*-algebra morphism $C^*\Gamma\to M$,
thus every $M$-module is a $C^*\Gamma$-module.
    The first part, regarding $[T_n]$ groups, follows from Theorem~\ref{thm:TnCor}. For the second part, regarding $(T_n)$ groups, use also Lemma~\ref{lem:comp}.
Note that if the map $C^*\Gamma\to M$ factors via $C_0^*\Gamma$ then all the relevant (co)homologies vanish. In particular, for every $p$, $L^p(M)^\Gamma=H^0(\Gamma,L^p(M))=0$.
By Lemma~\ref{lem:comp}, this is the case if $M^\Gamma=0$.
Assume now that for some $1\leq p< \infty$, $L^p(M)^\Gamma=0$.
By considering the Mazur map as in \cite{Olivier}, we assume as we may $p=1$, that is $(M_*)^\Gamma=0$. We conclude that $M_*$ is a $C_0^*\Gamma$-module, and by duality $M$ is a $C_0^*\Gamma$-module.
This concludes the proof.
\end{proof}

The fact that we consider here $\Gamma$-unitary von Neumann algebras is crucial. 
Consider, for example, the von Neumann algebra $\ell^\infty(\Gamma)=\ell^1(\Gamma)^*$, on which $\Gamma$ acts by automorphism, but it is not $\Gamma$-unitary.
Observe that for every infinite group $\Gamma$, 
$\ell^1(\Gamma)^\Gamma=0$ but $H^1(\Gamma,\ell^\infty (\Gamma))$ and $H_0(\Gamma,\ell^1 (\Gamma))$ are non-Hausdorff.
What can be said in general about homological properties of 
isometric actions on von Neumann algebras and the corresponding $L^p$-spaces?

\begin{conjecture}
    Let $\Gamma$ be a group with property $(T_n)$ and $M$ a von Neumann algebra such that the predual $M_\ast$ is endowed with an isometric linear action of $\Gamma$. 
    Then for $k\leq n$ and $1\leq p\leq 2$, the inclusion of the invariants yields an isomorphism in cohomology
    \[ H^k(\Gamma, L^p(M)^\Gamma) \cong H^k(\Gamma, L^p(M)). \]
    In particular, if $M_*^\Gamma=0$ then
    \[ H^k(\Gamma, L^p(M))=0. \]
\end{conjecture}

This conjecture in degree 1 is due to Mimura.
Then it is known to hold for $M=B(H)$ and $1<p\leq 2$ by \cite[Theorem 3(1)]{Mikala} and for any $M$ commutative by \cite{BFGM} ($1<p \leq 2$) and \cite{BGM} ($p=1$).
If $p>2$, the analogous conjecture is known not to hold, already in degree 1 and for $M$ commutative on which $\Gamma$ acts by automorphisms.
However, we expect it to hold assuming the existence of an invariant state.

\begin{conjecture} \label{conj:tracial}
Let $\Gamma$ be a group with property $[T_n]$ and $M$ a von Neumann algebra. 
 Assume $\Gamma$ acts on $M$ by automorphisms, preserving a faithful normal state.
Then for every $0<k\leq n$ and $1\leq p < \infty$, $H^k(\Gamma,L^p(M))=0$.
\end{conjecture}

In degree 1 and for $2\leq p < \infty$, this conjecture is proved by Marrakchi and de la Salle, \cite[Theorem 3.(2)]{Mikala}.
Note that for the profinite completion $\hat{\Gamma}$, we have 
$H^1(\Gamma,L^\infty(\hat{\Gamma}))\neq 0$, so the conjecture could not be extended to $p=\infty$\footnote{We thank Mikael de la Salle for sharing this observation with us.}.

\subsection{Group algebra reformulation of higher T}\label{subsec: group algebra reformulation}

In this subsection we discuss group algebra criteria for higher property~T.
We fix a free resolution $P_\bullet$ as in \S\ref{subsec: laplace operator}. Recall that $P_0=\bbC\Gamma$. Then the Laplace operator $\Delta_0\colon P_0\to P_0$ is given by an element in the group algebra $\bbC\Gamma$. 
The motivation for our discussion is the celebrated Ozawa criterion.

\begin{theorem}[\cite{OzawaT}] \label{thm:Ozawa}
    The group $\Gamma$ has property T if and only if there exist $\epsilon > 0$ and elements $x_1,\dots,x_j\in \bbC\Gamma$ such that 
    \[ \Delta_0(\Delta_0 -\epsilon) = \sum_{i=0}^j x_i^*x_i.\]
\end{theorem}

A criterion of a similar spirit was given in \cite{Bader-Nowak} for $[T_n]$,
with a caveat - the implication (1) $\Rightarrow$ (7) of Lemma~\ref{lem:basic} was yet not observed at the time of its publication.
In retrospect, we do get the following.

\begin{theorem}[\cite{Bader-Nowak}]\label{thm: bader-nowak for bracket T}
  The group $\Gamma$ has property $[T_n]$ if and only if 
  for every $1\leq k \leq n$ there exist $\epsilon > 0$ and elements $x_1,\dots,x_j\in M_{m_k}(\bbC\Gamma)$ such that 
    \[ \Delta_k -\epsilon = \sum_{i=0}^j x_i^*x_i.\]  
\end{theorem}

Here and in the sequel the Laplace operators are regarded as elements in the corresponding matrix algebras over $\bbC\Gamma$ via the bases of the free resolution $P_\bullet$. The number $m_k$ is the rank of the free module $P_k$. 

\begin{proof}
By Lemma~\ref{lem:basic}, the group $\Gamma$ has $[T_n]$ if and only if $\Delta_k$ is invertible as a matrix over $C^*\Gamma$ for every $1\leq k\leq n$. Hence the theorem follows from~\cite[Proposition 16]{Bader-Nowak}.
\end{proof}

We note that Theorem~\ref{thm: bader-nowak for bracket T} implies the original criterion of Ozawa.

\begin{proof}[Proof of Theorem~\ref{thm:Ozawa}]
    We will prove the non-trivial direction. 
    Assume $\Gamma$ has property T, hence by Theorem~\ref{thm: bader-nowak for bracket T},
    there exist $\epsilon > 0$ and elements $x_1,\dots,x_j\in M_{m_1}(\bbC\Gamma)$ such that 
    \[ \Delta_1 -\epsilon = \sum_{i=0}^j x_i^*x_i.\] 
    Then 
    \[ \Delta_0^2 -\epsilon \Delta_0 =d_1^*(\Delta_1 -\epsilon)d_1=d_1^*(\sum_{i=0}^j x_i^*x_i)d_1=\sum_{i=0}^j (x_id_1)^*(x_id_1)\]
    and we are done by Lemma~\cite[Lemma 14]{Bader-Nowak}.
\end{proof}

The following is a criterion of a similar spirit for $(T_n)$.
We regard the algebra $\bbC\Gamma$ as a unital subalgebra of every matrix algebra over $\bbC\Gamma$, by considering scalar matrices.

\begin{theorem_o}[stated in the introduction as Theorem~\ref{thm:grpalg}]
     The group $\Gamma$ has property $(T_n)$ if and only if 
  for every $0\leq k \leq n$ there exists $\epsilon > 0$ and elements $x_1,\dots,x_j\in M_{m_k}(\bbC\Gamma)$ such that 
    \[ \Delta_0(\Delta_k -\epsilon)\Delta_0 = \sum_{i=0}^j x_i^*x_i.\]   
\end{theorem_o}

In the upcoming discussion, towards the proof of Theorem~\ref{thm:grpalg},
we will freely use the terminology introduced in~\cite{Bader-Nowak}. 

\begin{remark}
If we start with a free $\bbQ\Gamma$-resolution $P_\bullet$, then 
the Laplace operators $\Delta_k$ are represented by matrices over $\bbQ\Gamma$. 
In Theorems~\ref{thm:Ozawa} and~\ref{thm:grpalg}, the matrix entries of the elements $x_i$ can then be chosen to be in $\bbQ\Gamma$ (cf.~\cite{Bader-Nowak}*{Corollary 15}). 
\end{remark}

\begin{lemma} \label{lem:inter}
    Let $A$ be a unital $*$-algebra that is archimedean has a trivial infinitesimal ideal. In particular, $A$ injects into its C*-envelope $C=C^*(A)$~\cite{Bader-Nowak}*{\S2.2}.
    Then every intermediate $*$-algebra $A\subset B\subset C$ is also archimedean and has a trivial infinitesimal ideal. We may identify naturally $C^*(B)=C$. 
    
    In particular, the norm on $B$ induced by the order extends the norm on $A$. 
    The order on $B$ might differ from the order on $A$, but for $a\in A$,
    $a\geq_A 0$ implies $a \geq_B 0$ and $a\geq_B 0$ implies that for every $\epsilon>0$, $a+\epsilon \geq_A 0$.
\end{lemma}

\begin{proof}
The entire lemma follows from its last sentence.
\cite[Proposition 9]{Bader-Nowak} gives the last sentence of the Lemma for $B=C$. In general, $A_+\subseteq B_+\subseteq C_+$ and the last sentence follows from the case $B=C$.
\end{proof}

\begin{lemma} \label{lem:invert}
    Let $A$ be a unital $*$-algebra that is archimedean and has a trivial infinitesimal ideal. Let $C=C^*(A)$ be its C*-envelope. 

    Let $e\in C_+$ be a central idempotent. Let $\pi$ be the $*$-morphism $A\hookrightarrow C \twoheadrightarrow eC$. Let $a_0\in A_+$ be an element such that $a_0e=a_0$ and $\pi(a_0)$ is invertible in $eC$. 
    
    Then $\pi(a)$ is invertible in $eC$ for every $a\in A_+$ if and only if there exists some  $\epsilon>0$ such that $a_0(a-\epsilon)a_0 \in A_+$.
    \end{lemma}

\begin{proof} We have the following equivalences.
    \[ \text{$\pi(a)$ is invertible}\Leftrightarrow \exists_{\epsilon>0}~\pi(a-\epsilon) \in eC_+\Leftrightarrow \exists_{\epsilon>0}~\pi(a_0)\pi(a-\epsilon)\pi(a_0) \in eC_+\]
If for some $\epsilon>0$, $a_0(a-\epsilon)a_0 \in A_+$ then $\pi(a_0)\pi(a-\epsilon)\pi(a_0) \in eC_+$, and $\pi(a)$ is invertible.
Assume $\pi(a)$ is invertible. Let $\epsilon>0$ be such that $\pi(a-2\epsilon) \in eC_+$.
Then $ae-2\epsilon e\in eC_+ \subset C_+$.
Set $B=A + \bbC e$. Lemma~\ref{lem:inter} implies that 
$ae-2\epsilon e + \epsilon \in B_+$.
Observing that $a_0B_+a_0 \subseteq A_+$ we obtain that 
$a_0(a-\epsilon)a_0 \in A_+$.
\end{proof}

\begin{proof}[Proof of Theorem~\ref{thm:grpalg}]
First note that 
$\Delta_0(\Delta_0 -\epsilon)\Delta_0 \in \bbC\Gamma_+$
implies that every unitary $\Gamma$-representation $V$ with $V^\Gamma=0$ has a spectral gap and $\Gamma$ has property~T.
So we may and will assume that $\Gamma$ has property~T.
We denote by $z$ the Kazhdan projection in $C^*\Gamma$ and set $e=1-z$.
Then $\Delta_0e=\Delta_0$ and it is invertible in $C_0^*\Gamma=eC^*\Gamma$.
By Lemma~\ref{lem:basic} $\Gamma$ has $(T_n)$ if and only if $\Delta_k$ is invertible as a matrix over $C_0^*\Gamma$ for every $1\leq k\leq n$. 
Thus the theorem follows from Lemma~\ref{lem:invert} with regard to the following setting: We take $A=M_{m_k}(\bbC\Gamma)$, identify $C^*(A)=M_{m_k}(C^*\Gamma)$, view $C^*\Gamma$ as a unital subalgebra of $C^*(A)$ and $e$ as a central idempotent of $C^*(A)$. Moreover, we take  $a_0=\Delta_0$ and $a=\Delta_k$. Compare~\cite[\S2.5]{Bader-Nowak}. 
\end{proof}

\begin{rem}
    For a $*$-algebra $A$ and a $*$-ideal $I\lhd A$ such that $A/I$ has a trivial infinitesimal ideal, we have $I_+=A_+\cap I$.
    Taking $A=M_{m_k}(\bbC\Gamma)$ and $I=M_{m_k}(\bbC_0\Gamma)$,
    it follows that the $x_i$'s in Theorem~\ref{thm:grpalg} could (in fact, must) be taken in $M_{m_k}(\bbC_0\Gamma)$.
    Similarly, the $x_i$'s in Theorem~\ref{thm:Ozawa}
    could (in fact, must) be taken in $\bbC_0\Gamma$.    
\end{rem}

\subsection{Permanence properties}

\subsubsection{Products}
The product of a group with property $[T_n]$ and a group with property $[T_{m}]$, does not have property $[T_{n+m}]$ in general. Similarly for $(T_n)$. However, we have the following easy fact.  
    
\begin{lemma} \label{lem:subadd}
    For $n,m\geq 1$, if $\Gamma$ has property $(T_n)$ and $\Lambda$ has property $(T_m)$, then \[H^k(\Gamma\times \Lambda, V)=0\] for $k\in\{0,\dots, n+m+1\}$ and for every unitary $\Gamma\times \Lambda$-representation $V$ with $V^\Gamma=V^\Lambda=0$. 
\end{lemma}

\begin{proof}
 The Hochschild-Serre spectral sequence for the extension $1\to \Gamma\to \Gamma\times\Lambda\to \Lambda\to 1$
has the $E^2$-term 
\[ E^{p,q}_2=H^p\bigl(\Gamma, H^q(\Lambda, V)\bigr).\]
If $0\le p+q\le n+m+1$, then $p\le n$ or $q\le m$. Say $p\le n$. The natural $C_0^\ast(\Gamma)$-action on $V$, which we have by the assumption $V^\Gamma=0$,  yields a $C_0^\ast\Gamma$-action on $H^q(\Lambda, V)$ extending the $\Gamma$-action. By applying Theorem~\ref{thm:criteria} and Lemma~\ref{lem: ring extension} to the $C_0^\ast\Gamma$-module $H^q(\Lambda, V)$ we obtain that $E_2^{p,q}=0$. If $q\le m$, then $E_2^{p,q}=0$ is immediately clear. 
So in the range $p+q\le n+m+1$ the spectral sequence collapses. 
\end{proof}

Considering both left and right multiplication, $C^*\Gamma$ and $C_0^*\Gamma$ could be seen as $\Gamma\times \Gamma$-modules.
The following observation follows from the above discussion.

\begin{lemma}
    If $\Gamma$ has $[T_n]$ then for every $j\leq 2n$, $H^j(\Gamma\times \Gamma,C^*\Gamma)=0$.
    A similar statement holds if $\Gamma$ has $(T_n)$ and $C^*\Gamma$ is replaced by $C_0^*\Gamma$.
\end{lemma}

Lemma~\ref{lem:subadd} is incorrect for $n=m=0$. Indeed, every group $\Gamma$ is $(T_0)$, so taking any $\Gamma$ which does not have T and taking $\Lambda$ trivial leads to a contradiction.
However, we have the following.

\begin{lemma} \label{lem:subadd0}
    Lemma~\ref{lem:subadd} still holds for $n=0$ and $m\geq 1$.
\end{lemma}

Note that if $\Gamma$ does not have property~T then $C_0^*\Gamma$ is not a unital algebra.

\begin{proof}
Arguing as in the proof of Lemma~\ref{lem:subadd}, the non-trivial statement to verify is that $H^0\bigl(\Gamma, H^{m+1}(\Lambda, V)\bigr)=0$.
We consider a free $\bbC\Lambda$-resolution $P_\bullet$ of $\bbC$ of finite type. Then $C^\bullet\defq \hom_{\bbC\Lambda}(P_\bullet, V)$ is a cochain complex of Hilbert spaces with 
$H^\bullet(\Lambda, V) = H^\bullet(C^\bullet)$.
Upon the identification of $C^k$ with $V^{m_k}$ as discussed in \S\ref{subsec: laplace operator},
each has the structure of a unitary $\Gamma$-representation.
By the assumption that $V^\Gamma=0$, we get $(C^k)^\Gamma=0$.
Each cocycle space $Z^{k+1}(\Lambda, V)\subset C^{k+1}$ is a $\Gamma$-invariant closed subspace, hence a unitary $\Gamma$-subrepresentation.
By Theorem~\ref{thm:TnCor}, $H^{m+1}(\Lambda,V)$ is Hausdorff, 
so the coboundary space $B^{m+1}(\Lambda, V)\subset C^{m+1}$
is a $\Gamma$-invariant closed subspace, hence unitary $\Gamma$-subrepresentation of $Z^{k+1}(\Lambda, V)$,
and accordingly, $H^{m+1}(\Lambda,V)$ has the structure of unitary $\Gamma$-representation.
We thus may consider the short exact sequence of 
unitary $\Gamma$-representations,
\[ 0 \to B^{m+1}(\Lambda,V) \to Z^{m+1}(\Lambda,V) \to H^{m+1}(\Lambda,V) \to 0\]
and its associated long exact sequence in bounded cohomology~\cite{frigerio}*{Section~1.5}.
This long exact sequence starts as 
\[ \cdots \to H_b^0(\Gamma,Z^{m+1}(\Lambda,V)) \to H_b^0(\Gamma,H^{m+1}(\Lambda,V)) \to H^1_b(\Gamma,B^{m+1}(\Lambda,V)) \to \cdots \]
It suffices to show that $H_b^0(\Gamma,Z^{m+1}(\Lambda,V))=0$ and $H^1_b(\Gamma, B^{m+1}(\Lambda,V))=0$.
To this end, we note that $H_b^0(\Gamma,Z^{m+1}(\Lambda,V))=H^0(\Gamma,Z^{m+1}(\Lambda,V))$ injects into $H^0(\Gamma,C^{m+1})=0$ and $H^1_b(\Gamma, B^{m+1}(\Lambda,V))=0$ because any bounded affine isometric action on a Hilbert space has a fixed point. 
\end{proof}

\subsubsection{Quotients}

Classical property T passes to quotients.
This follows immediately from the fact that every unitary representation of a quotient group yields a unitary representation of the original group, and so does every almost invariant sequence of unit vectors. 

\begin{rem}
    While property T passes to quotients, this is not the case with the higher analogues. For example, uniform lattices in the exceptional simple Lie group $F_4^{(-20)}$ have $(T_3)$ (see Theorem~\ref{thm:F4}) but 
Fournier-Facio has shown~\cite{fournierfacio} that they admit quotients that do not satisfy $(T_2)$.
\end{rem}

\subsubsection{Subgroups of finite index}

If $\Lambda$ is a finite index subgroup of $\Gamma$ or, more generally, $\Lambda$ is a cocompact lattice in a locally compact group~$G$, then the Shapiro lemma implies that property $[T_n]$ and property $(T_n)$ pass from $\Gamma$ or $G$ to $\Lambda$. 

The other direction is less obvious. The following result is~\cite{badsau}*{Theorem~6.21}.

\begin{theorem}
Property $[T_n]$ passes from a locally compact group to every of its cocompact lattices. In particular, property $[T_n]$ passes to subgroups of finite index. 
\end{theorem}

By same result from~\cite{badsau} property $(T_n)$ for a group~$G$ implies that the cohomology of every subgroup of~$G$ of finite index with coefficients in a weakly mixing unitary representation vanishes in degrees up to~$n$. But is less clear that property $(T_n)$ passes to subgroups of finite index in general. 

\subsection{Back to the finiteness assumption}
\label{subsec: back to fin}
We conclude this section by coming back to the finiteness property assumption $FP_\infty(\bbQ)$ which we assumed throughout.
In particular, the equivalence of the many properties in 
Theorem~\ref{thm:criteria} and the deduction of Theorem~\ref{thm:TnCor} were proved under this assumption.
However, we note that the proof of the following lemma does not need any finiteness assumption.

\begin{lemma}\label{lem: ring extension}
   Let $A$ be a unital algebra extending $\mathbb{C}\Gamma$
    and fix an integer $n$.
   If $H_k(\Gamma,A)=0$ for every $0\le k\leq n$, then $H_k(\Gamma,V)=0$ and $H^k(\Gamma, V)=0$ for every $A$-module $V$ and for every $0\le k\leq n$. 
\end{lemma}

\begin{proof}
The projective $A$-chain complex $A\otimes_{\bbC\Gamma} P_\bullet$ is exact up to degree~$n$. Thus, it is homotopy equivalent to a projective $A$-chain complex $Q_\bullet$ with $Q_0=\dots=Q_n=0$. Now use the homotopy equivalences $V\otimes_{\bbC\Gamma}P_\bullet\simeq V\otimes_A Q_\bullet$ and $\hom_{\bbC\Gamma}(P_\bullet, V)\simeq \hom_A(Q_\bullet, V)$.
 \end{proof}

Assume now that $\Gamma$ is finitely generated, but do not assume any further finiteness properties.
Note that Lemma~\ref{lem:nonT} and Remark~\ref{rem:nonT} still hold, thus we have that $H_k(\Gamma,C_0^*\Gamma)=0$ implies that $\Gamma$ has property T, thus the C*-algebra $C_0^*\Gamma$ is unital
 by Lemma~\ref{lem:kproj} and Convention~\ref{conv:kproj}.
 It follows from Lemma~\ref{lem: ring extension} that under merely finite generation assumption, 
 the property of having for every $k\leq n$, 
 $H_k(\Gamma,C_0^*\Gamma)=0$, implies $(T_n)$ and the entire variety  
of properties listed in Theorem~\ref{thm:criteria} and Theorem~\ref{thm:TnCor}.
This property is equivalent to $(T_n)$ under $FP_\infty(\bbQ)$,
but this equivalence is not known in general, which brings us back to 
Question~\ref{q:fin}.

\section{Lattices in semisimple groups}
\label{sec: lattices in semisimple groups}

In this section we turn to lattices in semisimple groups, which are our main source of examples for higher property~T. 

In \S\ref{sec:proof} we sketch the proofs of the main results about higher property~T of lattices in semisimple groups. 
We first revisit the proof of Theorem~\ref{thm:higher} for unitary coefficients. Then we focus on possible generalizations with Banach space coefficients. This leads to 
Theorem~\ref{thm:noncom} and, conditional on a spectral gap conjecture for certain rank~$1$ subgroups (Conjecture~\ref{conj:ai}), to a general higher property~T statement for super‑reflexive Banach coefficients. See Theorem~\ref{thm:mainBBBS}.

We then survey several related “below‑rank phenomena”: In~\S\ref{subsec: cohomological below-rank} we discuss cohomological below-rank phenomena as vanishing and non‑vanishing results for $L^p$-cohomology and torsion and homology growth. In~\S\ref{sec:geo} we discuss geometric below-rank phenomena as polynomial filling functions and Shapiro‑type results for non‑uniform lattices, Farb’s property FA$_n$ and Zimmer‑style rigidity, as well as uniform waist inequalities and high‑dimensional expansion. 

Finally, in \S\ref{sec:sd} we isolate low‑degree consequences in degrees 1 and 2, organizing a web of conjectures around two spectral gap conjectures (Conjecture~\ref{conj:SG} and~\ref{conj:SSG}) and its implications for property $\tau$, coamenable subgroups, invariant random subgroup and character rigidity, and stability and approximation problems.

The writing style in~\S\ref{sec: lattices in semisimple groups} is more of a survey, highlighting many conjectures and their interactions. 

\subsection{From Theorem~\ref{thm:higher} towards Conjecture~\ref{conj:SR} and Theorem~\ref{thm:noncom}} \label{sec:proof}

Theorem~\ref{thm:higher} was proved by the two authors of the present paper  in~\cite{badsau}.
The following is a rough sketch of the proof in~\cite{badsau}.

\begin{proof}[Sketch of the proof of Theorem~\ref{thm:higher}]
The proof goes along the following steps:
     \begin{enumerate}
     \item Use results from representation theory by 
     Zuckerman \cite{Zuckerman} and Borel-Wallach \cite{Borel-Wallach}
     (see also Vogan-Zuckerman~\cite{vogan+zuckerman}) saying that the cohomology of~$G$ vanishes in degrees below the rank for non-trivial irreducible unitary representations. 
     \item Every unitary representation has a direct integral decomposition into irreducible unitary representations. The direct integral decomposition is compatible with reduced cohomology by~\cite{Blanc}*{Th\'eor\`eme~7.2}.      
     As a consequence, one obtains the vanishing of the reduced cohomology of~$G$ for all unitary representations without invariant vectors in degrees below the rank. 
     \item By Shapiro Lemma, we get a similar vanishing result of the reduced cohomology for a cocompact lattice $\Gamma$ in $G$.
     \item By Lemma~\ref{lem:uptrick} the same vanishing holds for the (ordinary) cohomology of~$\Gamma$.
     \item Using Shapiro Lemma again, we deduce that such vanishing holds for the ordinary cohomology of~$G$.
     \item Finally, to deduce that a similar vanishing result applies to an arbitrary (maybe non-cocompact) lattice, a replacement for the Shapiro lemma is developed in~\cite{badsau} based on recent advances in geometric group theory. See the discussion after Theorem~\ref{thm: leuzinger+young}. This implies the same vanishing result for all lattices in $G$.
 \end{enumerate}
 This concludes the sketch of the proof of Theorem~\ref{thm:higher}. 
\end{proof}

\begin{rem} \label{rem:zigzag}
    In steps (1) and (2) we discussed the cohomology of $G$, then we \emph{zigged} to a cocompact lattice $\Gamma$ in (3) and (4) and \emph{zagged} back to $G$ in (5). Finally we \emph{zigged} to a lattice again in (6).
    This is due to the fact that Lemma~\ref{lem:uptrick} could not be applied to $G$ directly - not only that it requires the standing finiteness property assumption \ref{conv:ft}, but also the ultrapower of a continuous $G$-representation is a $G$-representation which is not continuous anymore.
\end{rem}

In this subsection we discuss generalizations -- both established and conjectural ones -- of Theorem~\ref{thm:higher} where we significantly extend the possible coefficients from Hilbert spaces to more general Banach spaces.  

We will start with the following remarkable theorem in degree~$1$, which was conjectured in~\cite{BFGM}. It was  proved recently by Oppenheim~\cite{Oppenheim} and de Laat--de la Salle~\cite{Sa-Le-at};  over a non-archimedean field it was proved earlier by Lafforgue~\cite{Lafforgue}.

\begin{theorem}[BFGM conjecture] \label{thm:Opp}
    Consider an isometric linear representation of a simple group $G$ of higher rank
on a super-reflexive Banach space $V$ with $V^G=0$.
Then $H^1(G,V)=0$.
The same holds for lattices in $G$.
\end{theorem}

In attempting a generalization to higher degrees, Conjecture~\ref{conj:ai} below was formulated in a recent joint work with Saar Bader and Shaked Bader~\cite{BBBS}.

Every simple root of the root system of $G$ gives rise to a conjugacy class of rank~$1$ simple subgroups of $G$.
We call these the \emph{standard rank~$1$  subgroups}.

\begin{conjecture}[\cite{BBBS}] \label{conj:ai}
For every isometric linear representation of a simple group~$G$ of higher rank on a super-reflexive Banach space~$V$, 
every standard rank 1 subgroup of~$G$ acts with no almost invariant vectors. 
\end{conjecture}

The main work of~\cite{BBBS} lies in proving the following theorem, which relates Conjecture~\ref{conj:ai} and  Conjecture~\ref{conj:SR}.

\begin{theorem}[\cite{BBBS}] \label{thm:mainBBBS}
Let $\mathbf{G}$ be a simple algebraic group of rank $r$ over a characteristic 0 local field $F$ and let $\Gamma<G=\mathbf{G}(F)$ be a lattice. 
   Let $V$ be a super-reflexive Banach space and $\Gamma\to B(V)$ be a linear isometric representation with no invariants, that is $V^\Gamma=0$.
   If Conjecture~\ref{conj:ai} is true then for every $j<r$, 
   \[ H^j(\Gamma,V)=0. \]

In short:
Conjecture~\ref{conj:ai} implies Conjecture~\ref{conj:SR}.
\end{theorem}

\begin{remark}\label{rem: proof strategy}
If one tries to apply the same proof strategy as for Theorem~\ref{thm:higher} sketched above, the first and second step break down. There is no decomposition into irreducibles in the world of Banach space representations. 
Note that the first and second step are easy and quite standard in the unitary case. 
The innovation of the proof of Theorem~\ref{thm:higher} lies mainly in steps (4), (5) and (6), which remain indispensable in the proof of Theorem~\ref{thm:mainBBBS}. 
The major innovation of~\cite{BBBS}, in which we prove Theorem~\ref{thm:mainBBBS} and its consequences Theorem~\ref{thm:noncom} and Corollary~\ref{cor:com}, is that we can avoid the results from Vogan-Zuckerman.
\end{remark}

Next we sketch the main ideas behind the proof of Theorem~\ref{thm:mainBBBS}. For more details we direct the reader to the forthcoming paper~\cite{BBBS}.

\begin{proof}[Sketch of the proof of Theorem~\ref{thm:mainBBBS}]

The proof goes by induction on the rank, where the result for a group follows from the corresponding results for the semisimple parts of its Levi subgroups.
As Conjecture~\ref{conj:SR} is formulated for simple groups and Conjecture~\ref{conj:semisimple} is not quite what we need, we will formulate another version, which will be amenable to a proof by induction.

\begin{prop} \label{prop:induction}
    We assume Conjecture~\ref{conj:ai}.
    Let $S$ be a semisimple group of rank~$r$. Let $S$ act by linear isometries on a super-reflexive Banach space~$V$ such that every simple factor has no almost invariant vectors.     
    Then $H^j(S,V)=0$ for every $j<r$ and $H^r(S,V)$ is Hausdorff.
\end{prop}

Applying this Proposition~\ref{prop:induction} to $S$ simple gives Theorem~\ref{thm:mainBBBS}.
Conversely, if there are no rank~$1$ factors in~$S$, Theorem~\ref{thm:mainBBBS} implies Proposition~\ref{prop:induction} using a version of Lemma~\ref{lem:subadd}.
If rank~$1$ factors do appear, we need be smarter.

\begin{prop} \label{prop:center}
    We assume Conjecture~\ref{conj:ai}. 
    Let~$L$ be a reductive group with semisimple part~$S$ of rank~$r$.
    Let $L$ act by linear isometries on a super-reflexive Banach space $V$ 
    such that the action of every simple factor of~$S$ has no almost invariant vectors in~$V$.     
    
    If the center of $L$ has no non-trivial invariant vectors, then $H^r(L,V)=0$.
\end{prop}

Proposition~\ref{prop:center} follows from Proposition~\ref{prop:induction} using a version of Lemma~\ref{lem:subadd0}, upon letting $\Lambda=S$ and letting $\Gamma$ be the center of $L$. 

For the induction proof of Proposition~\ref{prop:induction}, we use the $G$-action on the \emph{opposition complex} associated with the Tits building of $G$ -- a sufficiently highly connected simplicial complex with stabilizer subgroups to which we can apply the inductive hypothesis. 

The $n$-simplices of the opposition complex consist of ordered pairs of opposite $n$-simplices in the Tits building. 
The dimension of the opposition complex is thus the same as the one of the Tits building, namely~$r-1$. There is an analogue of the Solomon-Tits theorem for the opposition complex saying that it is 
$(r-2)$-connected~\cite{heydebreck}. Further, the stabilizers of the $k$-dimensional simplices with respect to the natural $G$-action on the opposition complex are Levi subgroups with semisimple rank~$r-k-1$.
The proof of the vanishing of cohomology part of Proposition~\ref{prop:induction} follows by mathematical induction using a cohomological induction argument.
Finally, the Hausdorffness of the cohomology at the rank follows by an ultrapower argument.
It is here, in taking an ultrapower, that Conjecture~\ref{conj:ai} is used to show that we gain no further invariants in the limiting process.

There is a delicate issue that we have to address at this point.
While forming the induction argument alluded to above, we have to guarantee that the conditions of Proposition~\ref{prop:center} are satisfied when applied to stabilizers of the opposition complex. This is done using a suitable version of the Howe-Moore theorem, which is at our disposal when working with groups over local fields. But when performing the ultrapower argument and when acting on the opposition complex, continuity breaks and we better work with countable subgroups.
We thus need to zig-zag between these two setups.
We do this in a way similar to the proof of Theorem~\ref{thm:higher}; see Remark~\ref{rem:zigzag}.
However, in order for the opposition complex associated with a lattice to have the needed dimension $r-1$, we need the lattice to be an arithmetic group of the same split rank as $G$, and in particular, it could not be a cocompact lattice. 
We thus use in an essential way the subtle version of the Shapiro Lemma which was used in step (6) of the proof of Theorem~\ref{thm:higher}.
With such a zig-zag argument we conclude the proof of Proposition~\ref{prop:center}, and this concludes the sketch of the proof of Theorem~\ref{thm:mainBBBS}. 
\end{proof}

To conclude, we have achieved a conditional proof of Conjecture~\ref{conj:SR} based on Conjecture~\ref{conj:ai}. Next we want to discuss important cases of Conjecture~\ref{conj:SR} that we can prove unconditionally. 

To this end, 
we regard the statement about~$V$ in  Conjecture~\ref{conj:ai}, which is about all super-reflexive Banach spaces, as a property of any specific space $V$ or any class of such spaces.

\begin{prop} \label{prop:list}
    The following classes of spaces are known to satisfy Conjecture~\ref{conj:ai}:
    \begin{itemize}
        \item Hilbert spaces with the Hilbert norm.
        \item Hilbert spaces with an  arbitrary equivalent norm.
        \item $L^p$-spaces for  $1<p<\infty$.
        \item non-commutative $L^p$- spaces for $1<p<\infty$.
    \end{itemize}
\end{prop}

For unitary representations, the conjecture follows by considering a subgroup of $G$ of the form $\SL_2(\mathbb{R})\ltimes E$, where $E$ is a finite dimensional representation of $\SL_2(\mathbb{R})$, and using its relative property T to find an $E$-fixed point, which gives a contradiction by the Howe-Moore theorem.
For an equivalent norm on a Hilbert space, following an idea by Shalom, one may use unitarization with respect to amenable subgroups.
This was carried out by Glasner and Gorfine and will appear in their forthcoming work~\cite{GG}.
For the other spaces in the list above, the proof is by comparison to the unitary case, see~\cite{BFGM} and~\cite{Olivier}.
We expect that the class above could be enlarged dramatically by the method of interpolation.

As a direct corollary of the proof of Theorem~\ref{thm:mainBBBS} and Proposition~\ref{prop:list} we obtain the following. 

\begin{theorem}[\cite{BBBS}]\label{thm: conj for list of spaces}
    Conjecture~\ref{conj:SR} holds for the list of spaces in Proposition~\ref{prop:list}.
\end{theorem}

The theorem contains the statements of Theorem~\ref{thm:noncom} and its Corollary~\ref{cor:com} apart from the edge case~$p=1$. We now repeat Theorem~\ref{thm:noncom} from the introduction for convenience. 

\begin{theorem_o}[stated in the introduction as Theorem\ref{thm:noncom}]
Let $\Gamma$ be a lattice in a simple Lie group $G$ of rank $r$ and $M$ a von Neumann algebra. Then for $k\leq r-1$ and $1\leq p < \infty$, for every 
linear isometric action of $\Gamma$ on $L^p(M)$
the inclusion $L^p(M)^\Gamma\hookrightarrow L^p(M)$ induces isomorphisms
    \[ H^k(\Gamma, L^p(M)^\Gamma) \cong H^k(\Gamma,L^p(M)). \]
    Moreover, $H^r(\Gamma,L^p(M))$ is Hausdorff.
    For $p=\infty$ this holds under the extra assumption that $M$ is $\Gamma$-unitary.
    Similar statements hold for $G$.
\end{theorem_o}

\begin{proof}[Sketch of the proof of Theorem~\ref{thm:noncom}]
  Because of Theorem~\ref{thm: conj for list of spaces} and the remark thereafter it remains to deal with the case~$p=1$, which is in large parts like the proof of Theorem~\ref{thm:mainBBBS} but needs extra care as the coefficients are not super-reflexive.  
  
  However, the obvious analogue of Conjecture~\ref{conj:ai} does hold, due to \cite{Olivier}. Translating the proof of Theorem~\ref{thm:mainBBBS} goes rather smoothly with only one crucial difference - 
when proving Proposition~\ref{prop:center} using a version of Lemma~\ref{lem:subadd0}, in showing that $H^1_b(C, B^r(S,V))=0$ we cannot use the fact that the space $B^r(S,V)$ is super-reflexive. Here~$V$ denotes the coefficients in question. Rather, we prove that this space is L-embedded and we invoke the fixed point theorem of~\cite{BGM} to conclude the required cohomological vanishing.
\end{proof}

\subsection{Cohomological below-rank phenomena}\label{subsec: cohomological below-rank}

\subsubsection{Borel's stability and $K$-theory}\label{subsec: Borel stability}

Borel’s stability theorem asserts that, for a fixed degree, the rational cohomology of arithmetic groups such as $\SL_n(\bbZ)$ stabilizes as $n\to\infty$, and coincides with the corresponding stable cohomology of the associated real Lie group. Specifically, Borel proves the following result for $\SL_n$; similar results hold for other classical groups.

\begin{theorem}[Borel~\cite{borel-stable}]\label{thm:borel-stable}
For every integer $i\ge 0$, the maps induced by the standard inclusions
The restriction map 
\[H^i_c\bigl(\SL_n(\bbR),\bbC\bigr)\to H^i\bigl(\SL_n(\bbZ),\bbC\bigr)\]
is an isomorphism for $0\le i<(n-1)/4$. 
\end{theorem}

Via this result, Borel identified the stable rational cohomology as an exterior algebra on explicit generators in odd degrees, now known as the \emph{Borel classes}. This result plays a central role in the computation of rational algebraic $K$-theory of rings of integers: via the comparison between group cohomology and algebraic $K$-theory, Borel’s stability allows one to determine the ranks of $K_i(\bbZ)\otimes \bbQ$ and, more generally, of $K_i(\calO_F)\otimes \bbQ$ for number fields~$F$. 

Our Theorem~\ref{thm:semisimple Lie without T} from~\cite{badsau} significantly improves Borel's stability range and allows for arbitrary unitary coefficients. 

This yields new computations of the low-degree cohomology of arithmetic groups; it does not yield new computations of algebraic $K$-theory as the slope of the stability range does not matter here. For the trivial representation and congruence lattices, Li-Sun independently provide a similar improvement of Borel's stability range~\cite{Li-Sun}*{Theorem~1.8}. Li-Sun's paper relies heavily on the work of Franke~\cite{Franke}, which plays no role in our approach. 

The following result of Borel-Yang was the key to their solution of the rank conjecture in algebraic K-theory. The paper by Borel-Yang is heavily based on the work of Blasius-Franke-Grunewald~\cite{blasius+franke+grunewald}, which itself relies on Franke's work~\cite{Franke}.

\begin{theorem}[Borel-Yang~\cite{Borel-Yang}] \label{thm:Borel-Yang}
The restriction map 
\[ H^\ast_c\bigl(\SL_n(\bbR),\bbC\bigr)\to H^\ast\bigl(\SL_n(\bbQ),\bbC\bigr)\]
is an isomorphism in all degrees. More generally, if $k$ is a number field and $\mathbf{G}$ is a connected, simply connected, almost simple $k$-algebraic group, then the restriction map 
\[ H_c^*\bigl(\mathbf{G}(k\otimes\bbR),\bbC\bigr)\to H^*\bigl(\mathbf{G}(k),\bbC\bigr) \]
is an isomorphism in all degrees. 
\end{theorem}

The correct generalization of Borel-Yang's theorem to unitary coefficients is Theorem~\ref{thm:adelic} in~\S\ref{subsec:poly} and needs the adelic framework. 

We refer to~\cite{badsau}*{Theorems F and G} for further generalizations of Borel-Yang's theorem. 
In fact, the entire unitary cohomology theory of $\mathbf{G}(k)$ is determined by that of $\mathbf{G}(\mathbb{A}(k))$, at least when the $k$-rank of $\bfG$ is at least 2.  Although $\bfG(k)$ is not of type I and a complete understanding of its unitary dual is out of reach, Theorem~7.9 and Corollary~7.10 of~\cite{badsau} determine the cohomological unitary dual of $\bfG(k)$ completely.

Also these results from~\cite{badsau} are proved using the strategy outlined in \S\ref{sec:proof}, with a heavy reliance on the recent advances in geometric group theory discussed in~\S\ref{subsec:poly}. We think it would be very interesting to understand better the connection of our methods to Franke's work~\cite{Franke}. 

\subsubsection{Gromov's conjecture on $L^p$-cohomology} \label{sec:LP}

In this subsection we focus on specific Banach representations of a group~$\Gamma$, namely the regular ones, $L^p(\Gamma)$, for $p\in [1,\infty)$. The resulting (continuous) group cohomology with coefficients in $L^p(\Gamma)$ is called the \emph{$L^p$-cohomology} of~$\Gamma$ -- a subject that owes a lot to the work of Pansu and Gromov, see for instance \cite{Pansu} and \cite{Gromov93}*{\S8}. 

We remark that for $p=2$ the $L^p$-cohomology is a module over the von Neumann algebra $L\Gamma$ and the availability of Murray-von Neumann dimension is a valuable tool in this case.
However, we are interested in arbitrary $p$ and it was shown in \cite{Monodlp} that a reasonable dimension theory is not expected for $p>2$, so our discussion below is dimension free~\footnote{Nonetheless, note that an entropy-type dimension for sofic groups and for $p<2$ was suggested in~\cites{Hayes1,Hayes2}.}.

The first $L^p$-cohomology is well studied. See~\cites{pansu-cohom, BP, Yu, Nica, BFGM, cornulier+tessera-cohom_banach, cornulier+tessera-contracting}.
For every $p$, the first $L^p$-cohomology  is Hausdorff if and only if $\Gamma$ is non-amenable.
If $\Gamma$ has property~T and $p\in [1,2]$, then $H^1(\Gamma, L^p(\Gamma))=0$ by \cite{BFGM} (for $p>1$) and \cite{BGM} (for $p=1$).
A similar statement for property~T groups in general fails for $p>2$.
%and the failure might be understood as a sign of a less robust form of property~T. 
Pansu proved that a lattice in a Lie group of rank~$1$ has $H^1(\Gamma, L^p(\Gamma))\ne 0$ for sufficiently large~$p$~\cite{pansu-cohom}. In particular, property $[T_2]$, which is satisfied by $\Sp(n, 1)$ and its cocompact lattices, is not robust enough to ensure that the $L^p$-cohomology vanishes for every $p\in [1,\infty)$. 
Later, Bourdon-Pajot generalized Pansu's result to all non-elementary Gromov-hyperbolic groups~\cite{BP}, see also~\cite{Yu, Nica}. 
The situation for a lattice in higher rank simple Lie groups is different: There we have $H^1(\Gamma, L^p(\Gamma))=0$ for every $p\in [1,\infty)$ by an unpublished result of Pansu and~\cite[Theorem~B]{BFGM}.  

%It might be interesting to investigate the class of groups with property~T whose $L^p$-cohomology vanishes for every $p\in [1,\infty)$ further. Is there an interesting robust form of property~T that implies that a group belongs to this class? To obtain at least one example of a property~T group $\Gamma$ in this class that is not a higher rank lattice, take $\Gamma=\SL_3(\bbZ)\times \Lambda$ where $\Lambda$ is an arbitrary property~T group: Then the $E_2^{pq}$-term of the Lyndon-Hochschild-Serre spectral sequence for $1\to \Lambda\to\Gamma\to \SL_3(\bbZ)$ vanishes in the range $p+q=1$ since $H^1(\SL_3(\bbZ), L^p(\Gamma))=0$ by~\citelist{\cite{BFGM}*{Theorem~B}\cite{BGM}}, hence $E_2^{01}=0$, and, obviously, $E_2^{10}=0$.    

The $L^p$-cohomology in higher degrees and for $p\ne 2$ is much less understood. 
Using Lemma~\ref{lem:subadd}, one easily sees that for a product of finite type non-amenable groups, $\Gamma=\Gamma_1\times\cdots \Gamma_n$, the $L^p$-cohomology vanishes in every degree $j<n$. It is also not hard to see that in degree $n$ it equals the $L^p$-tensor product of the first $L^p$-cohomologies of its factors.
See \cite[p.~252]{Gromov93} for a de Rham theoretic proof.
This fact motivated Gromov to conjecture a cohomological vanishing below-rank for symmetric spaces and buildings, see \cite[p.~253]{Gromov93}. 
Moreover, Gromov conjectured that at the rank the $L^p$-cohomology is Hausdorff and that it doesn't vanish for $p$ large enough, in fact for all $p>1$ in the building case.
This conjecture is now fully confirmed.

\begin{theorem}[Gromov's conjecture] \label{thm:Gromov}
Let $G$ be a non-compact semisimple group over a local field $F$ with finite center and of rank $r$. Then:
\begin{enumerate}
\item $H_c^i(G, L^p(G))=0$ for every $0\le i < r$ and $p\in [1,\infty)$; 
\item $H_c^r(G, L^p(G))$ is Hausdorff for every $p\in [1,\infty)$ and it doesn't vanish for every sufficiently large~$p$, and for every $p$ if $F$ is non-archimedean;
\item  $H_c^i(G, L^p(G))=0$ for every $r< i$ and sufficiently large~$p$, and for every $p$ if $F$ is non-archimedean. 
\end{enumerate}    
Moreover, similar statements hold for all irreducible lattices in $G$, with the possible exceptions for non-uniform lattices of the non-vanishing at the rank in (2) and vanishing above the rank and large $p$ in (3).
\end{theorem}

\begin{proof}
Over archimedean fields, both the non-vanishing in Statement~(2) and the vanishing in Statement~(3) are due to Bourdon-R\'emy~\cite{BR23}*{Theorem~A}.
Over non-archi\-medean fields, the non-vanishing in Statement~(2) is due to \cite[Corollary 1.2]{Lopez} and the vanishing in Statement~(3) follows by cohomological dimension.
In both cases, and both for the groups and their lattices, Statement~(1) and the Hausdorffness in Statement~(2) is due to \cite{BBBS}, being a special case of Corollary~\ref{cor:com}. 
For cocompact lattices, the full result follows by Shapiro Lemma and $L^p$-induction.
\end{proof}

\begin{conjecture}
    Theorem~\ref{thm:Gromov} holds also for non-uniform irreducible lattices.
\end{conjecture}

Next, we consider $L^p(G)$ as a $G\times G$ module, for both the left and the right actions.
Then the semisimple version of Corollary~\ref{cor:com} gives the following.

\begin{theorem}
    Let $G$ be a non-compact semisimple group over a local field $F$ with finite center and of rank $r$. Then for every $p\in [1,\infty)$, $H_c^i(G\times G, L^p(G))=0$ for every $0\le i < 2r$ and it is Hausdorff for $i=2r$.
    A similar result holds for all irreducible lattices in $G$.
\end{theorem}

\begin{example}
    For $\Gamma$ hyperbolic group, following the work of Bourdon-Pajot \cite{BP}, Nica constructed for every $p$ large enough a non-trivial element in $H^1(\Gamma,H^1(\Gamma,\ell^p\Gamma))$, see \cite[Theorem 8]{Nica}.
Using the Lyndon-Hochschild-Serre spectral sequence
we have the identification  $H^1(\Gamma,H^1(\Gamma,\ell^p\Gamma))=H^2(\Gamma\times \Gamma,\ell^p\Gamma)$ and we interpret the Nica cocycle accordingly.
In particular, we get that $H^2(\Gamma\times \Gamma,\ell^p\Gamma)\neq 0$.
\end{example}

\begin{conjecture}
       Let $G$ be a non-compact simple group over a local field $F$ with finite center and of rank $r$. Then for every $p$ large enough $H_c^{2r}(G\times G, L^p(G))\neq 0$.
    A similar result holds for all lattices in $G$.
\end{conjecture}

\subsubsection{Torsion growth} \label{sec:other}

The first $\ell^2$-Betti number of a property~T group~$\Gamma$ vanishes. 
If $\Gamma$ is, in addition, residually finite and $(\Gamma_i)_{i\in\bbN}$ is a residual chain, then L\"uck's approximation theorem implies that the sequence of Betti numbers $b_1(\Gamma_i)=\dim_\bbQ H_1(\Gamma_i;\bbQ)$ grows sublinearly in the index $[\Gamma:\Gamma_i]$. 

A promising attempt to study asymptotic homology growth via dynamical means originates from the work of Ab\'ert-Nikolov~\cite{abert+nikolov}. They show that the rank gradient of $\Gamma$ along a residual chain~$(\Gamma_i)$ is equal to the cost of the measured orbit equivalence of the natural action of~$\Gamma$ on the projective limit~$\varprojlim \Gamma/\Gamma_i$. In particular, $\limsup_{i\to\infty} b_1(\Gamma; k)/[\Gamma:\Gamma_i]$ is bounded from above by the cost of $\Gamma\acts\varprojlim \Gamma/\Gamma_i$ minus~$1$ for every field~$k$, including finite fields $k=\bbF_p$. 

How does this connect to property~T? Hutchcroft-Pete~\cite{hutchcroft+pete} show that every countable infinite property~T group admits a free measure-preserving action with cost~$1$. If Gaboriau's fixed price conjecture 
would hold
for property~T groups, the cost would be independent
of the specific free measure-preserving action. When compared with the projective limit action for a residually finite property~T group~$\Gamma$, we would obtain that $\lim_{i\to\infty} b_1(\Gamma; k)/[\Gamma:\Gamma_i]=0$. The fixed price conjecture remains open in the generality of property~T groups. 
\begin{question}
Let $\Gamma$ be a residually finite group, and let $(\Gamma_i)_{i\in\bbN}$ be a residual chain. If $\Gamma$ has property~T, does $\Gamma$ have vanishing rank gradient and does 
\[\lim_{i \to\infty}
  \frac{\dim_{\bbF_p} H_1(\Gamma_i;\bbF_p)}
       {[\Gamma:\Gamma_i]}=0\]
hold? If $\Gamma$ has property~$(T_n)$ or $[T_n]$, does the corresponding $\bbF_p$-homology gradient vanish up to degree~$n$? For the homology gradient over a field of characteristic zero, this follows immediately from L\"uck's approximation theorem provided~$\Gamma$ satisfies suitable finiteness conditions, e.g.~being of type $F_{n+1}$. 
\end{question}

For higher-rank lattices we know more. 
Ab\'ert-Gelander-Nikolov~\cite{abert+gelander+nikolov} show that right-angled lattices in higher-rank simple Lie groups have vanishing rank and first $\bbF_p$-homology gradient. This includes all lattices of $\bbQ$-rank at least~$2$ but also some uniform lattices. Ab\'ert-Bergeron-Fraczyk-Gaboriau~\cite{abert+bergeron+fraczyk+gaboriau} show that the $\bbF_p$-homology and homology torsion gradients of a higher-rank lattice vanish in a range below the rational rank of~$\Gamma$. Then Fraczyk-Mellick-Wilkens~\cite{fraczyk+mellick+wilkens} showed that all higher rank lattices satisfy Gaboriau's fixed price conjecture and thus the expected vanishing rank and homology gradients in degree~$1$ hold. 

The following is the most far reaching conjecture that subsumes all the above, as well the idea of Benjamini-Schramm continuity of homology from~\cite{7samurai}. It was formulated in~\cite{li+loeh+moraschini+sauer+uschold}*{Conjecture~1.15}.  
    
\begin{conjecture}\label{conj: torsion cohomology conjecture}
Let $G$ be a semisimple Lie group with finite center and without compact factors. Let $r$ be the real rank of~$G$. 
Let $(\Gamma_i)_{i\in \bbN}$ be a sequence of irreducible lattices in~$G$ whose covolumes tend to infinity. Then 
\[\lim_{i \to\infty}
  \frac{\dim_{\bbF_p} H_n(\Gamma_i;\bbF_p)}
       {\vol(G/\Gamma_i)}=0
       \text{ and }
       \lim_{i \to\infty}
  \frac{\log \# \tors H_n(\Gamma_i;\bbZ)}
       {\vol(G/\Gamma_i)}=0
\]
for every~$n\in \{0,\dots, r-1\}$.
\end{conjecture}

Here are two wildly speculative attempts to tackle this conjecture. 
First, one might try to establish a fixed price conjecture for higher property~T and the higher-dimensional cost, which is ongoing work of Ab\'ert-Gaboriau-Tanushevski, and a higher-dimensional analogue of the theorem of Pete-Hutchcroft. We do not expect that higher property~T is sufficient but that specific facts about lattices are needed. However, one might speculate how far just higher property~T carries. 
Second, one might try to establish a fixed price property for the measured embedding dimension and measured embedding volume defined by Li-L\"oh-Moraschini-Sauer-Uschold~\cite{li+loeh+moraschini+sauer+uschold} and use their bounds~\cite{li+loeh+moraschini+sauer+uschold}*{Theorem~1.2} for the homology (torsion) gradients.

\subsection{Geometric below-rank phenomena} \label{sec:geo}

\subsubsection{Polynomial cohomology and polynomiality of filling function below the rank} \label{subsec:poly}

The proofs of Theorems~\ref{thm:higher} and~\ref{thm:mainBBBS} for non-uniform lattices have to solve the following problem. How to transfer cohomological information about the ambient Lie group~$G$ to the lattice and back in the absence of a Shapiro lemma that generally is only available for uniform lattices?

The idea in~\cite{badsau} to circumvent this difficulty is to use \emph{polynomial cohomology} as an auxiliary cohomology. Polynomial cohomology $H_{pol}^\ast(\Gamma, V)$ is defined using cochains in the bar complex that satisfy a polynomial growth constraint with respect to the word length. 
The comparison of polynomial and usual cohomology is via geometric group theory. 

The $d$-th \emph{filling function} $F_d^\Gamma(n)$ at $n>0$ of a group~$\Gamma$ bounds the minimal (combinatorial) volume of a $d$-chain bounding a $(d-1)$-cycle of volume at most~$n$ in a classifying space of finite type.  
The following is proved in~\cite{badsau}*{Proposition~6.14}. 

\begin{theorem}
Let $\Gamma$ be a group whose homological filling functions $F_1^\Gamma, \dots, F_d^\Gamma$ are polynomial. Then the comparison map \[H_{pol}^k(\Gamma, V)\to H^k(\Gamma, V)\] 
is an isomorphism for $0\le k\le d$. 
\end{theorem}

The following deep result by Leuzinger-Young~\cite{leuzinger+young} allows us to apply the previous theorem. 

\begin{theorem}[Leuzinger-Young]\label{thm: leuzinger+young}
The filling functions of an irreducible lattice in a connected semisimple Lie group~$G$ with finite center and without compact factors are polynomial below the rank of~$G$.  
\end{theorem}

In contrast, in the degree of the rank, the filling function of a non-uniform lattice is exponential. The filling functions of a uniform lattice are polynomial in all degrees. 

Morally, a kind of Shapiro lemma for the polynomial cohomology with unitary or Banach space coefficients is proved in~\cite{badsau}. Via the two previous theorems, we obtain a suitable Shapiro lemma for non-uniform lattices below the rank. 

In the S-arithmetic case and in positive characteristic, the equivalent of Theorem~\ref{thm: leuzinger+young} is not available. Bestvina-Eskin-Wortman~\cite{bestvina+eskin+wortman} show polynomiality of non-uniform lattices in S-arithmetic groups not below the $S$-rank but up to a degree bounded by the number of primes in~$S$. Further, Sauer-Weis~\cite{sauer+weis} show polynomiality below the rank (instead of the $S$-rank). Depending on the number of primes in~$S$, one or the other result might be stronger. 

\begin{conjecture}\label{conj: S-arithmetic filling}
    The analogous statement of Theorem~\ref{thm: leuzinger+young} holds in the S-arithmetic case in arbitrary characteristic. 
\end{conjecture}

From the techniques in~\cite{badsau} and \cite{BBBS} one can conclude the following. 

\begin{theorem}
Conjecture~\ref{conj: S-arithmetic filling} implies the validity of  Conjecture~\ref{conj:semisimple} for S-arithmetic lattices with property T
and unitary representations, and for general super-reflexive representations, assuming also Conjecture~\ref{conj:ai}.
\end{theorem}

One could try to generalize Harder's method in~\cite{harder} to prove Theorem~\ref{def:higherT} in positive characteristic. See also Remark~\ref{rem: harder}. The following theorem shows an alternative route to Theorem~\ref{thm:higher} in positive characteristic, assuming Conjecture~\ref{conj: S-arithmetic filling}. We refer to our forthcoming work~\cite{BBBS} for details. 

\begin{theorem}
Conjecture~\ref{conj: S-arithmetic filling} in the arithmetic case and for positive characteristic implies the validity of Theorem~\ref{thm:higher} for positive characteristic. 
\end{theorem}

Regarding the property T assumption, see the discussion in \S\ref{sec:deg1} below.
Finally, we remark that in the case where the number of inverted primes in~$S$ goes to infinity the range gap between the theorems of Leuzinger-Young and Bestvina-Eskin-Wortman becomes irrelevant. In the limiting, adelic case one can prove, for example, the following clean statement~\cite{badsau}, which  generalizes work of Borel and Yang. 

\begin{theorem} \label{thm:adelic}
Let $k$ be a number field and $\mathbb{A}(k)$ the ring of adeles of~$k$.
Let $\mathbf{G}$ be a connected, simply connected, almost simple $k$-algebraic group
and let $V$ be a unitary representation of the adelic group $\mathbf{G}(\mathbb{A}(k))$.
Then the restriction map
\[ \res:H_c^*(\mathbf{G}(\mathbb{A}(k)),V)\to H^*(\mathbf{G}(k),V) \]
is an isomorphism in all degrees.
\end{theorem}

Finally, we would like to mention a very recent development in the study of polynomial cohomology: L{\'o}pez Neumann and Paucar~\cite{neumann+paucar} used polynomial cohomology and ideas in~\cite{badsau} to prove interesting new invariance results for Betti numbers of nilpotent groups in the setting of quantitative measure equivalence. Further, they develop techniques to deal with unitary induction for non-uniform lattices in Lie groups of rank~$1$. See also the remark below Theorem~\ref{thm:F4}.  

\subsubsection{Farb's property FA$_n$ and speculations on the Zimmer program} \label{sec:Farb}
A group satisfies Serre's property FA if each of its actions on a tree has a fixed point.
It is well known that property T groups satisfy Serre's property FA.
In \cite{farb2009group} Farb makes the following generalization.

\begin{defn}
    A group $\Gamma$ has property FA$_n$ if any
cellular $\Gamma$-action on any $n$-dimensional CAT(0) cell complex has a global fixed point.
\end{defn}

Equivalently, $\Gamma$ has FA$_n$ if it cannot be written non-trivially as a non-positively curved $n$-dimensional complex
of groups.
Farb proved that various non-uniform lattices satisfy property FA$_{r-1}$, where $r$ is their rank,
and he asked whether every lattice in a simple Lie group of rank $r$ satisfies property FA$_{r-1}$.
Recently Fraczyk and Lowe \cite{fraczyklowe} showed that cocompact lattices in $\SL_n(\mathbb{R})$ satisfy property FA$_{\lfloor\frac{n}{8}\rfloor-1}$
and cocompact lattices in $F_4^{(-20)}$
satisfy property FA$_2$.
These are the first results concerning property FA$_n$ for cocompact lattices and $n>1$.
In the forthcoming paper~\cite{BBBS} we prove Farb's conjecture.

\begin{theorem}[Farb's conjecture, \cite{BBBS}] \label{thm:Farb}
   Lattices in simple groups of rank $r$ satisfy FA$_{r-1}$. 
\end{theorem}

This theorem follows from the following general theorem, together with the case $p=1$ of Corollary~\ref{cor:com}.

\begin{theorem}[\cite{BBBS}] \label{thm:Farb2}
Let $\Gamma$ be a group and fix $n\in \mathbb{N}$.
Assume that for every action of $\Gamma$ on a set $X$ with no finite orbits, for every $i< n$, $H^i(\Gamma,\ell^1(X))=0$.
Then any cellular action of $\Gamma$ on a contractible cellular complex of dimension less than $n$ has a finite orbit.
In particular, $(T_n)_{L^1}$ implies FA$_n$.
\end{theorem}

The proof of this theorem is by a certain diagram chasing over the double complex whose $p,q$-coordinate is the space of maps from $\Gamma^p$ to the space of $q$-chains of the given cellular complex, along with their $\ell^1$-completions.
We note that Fraczyk and Lowe conjectured that $(T_n)$ implies FA$_n$.

Theorems~\ref{thm:Farb} and \ref{thm:Farb2} fit Zimmer's philosophy whose best-known representative is the following.

\begin{conjecture}[Zimmer's Conjecture]
Every diffeomorphic action of a lattice in a simple group of rank $r$ on a compact manifold of dimension less than $r$ factors via a finite quotient.
\end{conjecture}

This conjecture is now proved in all cases, due to the work of Brown, Fisher, Hurtado and others.
In view of the above, we make the following.

\begin{conjecture} \label{conj:TZimmer}
    Every diffeomorphic action of a property $(T_n)_{L^1}$ group on a compact manifold of dimension less than or equal $n$ factors via a finite quotient.
\end{conjecture}

\subsubsection{Waist inequalities below the rank} \label{subsec:waist}

Waist inequalities predict 
the existence of a fiber with 
a uniform lower complexity bound 
for a family of maps to some Euclidean spaces, or more generally, manifolds. 
Gromov's famous waist inequality~\cite{gromov-waist} for the sphere $S^n$ says that the maximal volume of a fiber of a (generic) map from $S^n$ to $d$-dimensional Euclidean space is at least the $(n-d)$-dimensional volume of an equator sphere $S^{n-d}$. The focus of the following discussion is less on explicit lower complexity bounds but allows families of maps where the domain manifold varies. A typical example is to consider all finite covers of a fixed manifold in the domain. 
In this setting waist inequalities can be understood as Riemannian versions of higher-dimensional topological expanders. 

In this subsection we present a conjectural below-rank phenomenon for higher rank lattices and relate it to higher property~T. 
A differential-geometric approach to waist inequalities was developed by Fraczyk and Lowe~\cite{fraczyklowe}. Our account is influenced by the connection to higher property~T. 
See Figure~\ref{fig1} for a diagram illustrating much of the content of this subsection and the previous one.

\begin{figure}[ht] 
\centering
\begin{tikzcd}[scale=0.8]
\text{Zimmer Conjecture} & &\mathrm{FA}_n &\\
&{{[}\mathrm{T}_n{]}}_{L^1}\ar[d,Rightarrow]
\ar[ddl, decorate, decoration={snake, amplitude=0.3mm, segment length=1.5mm}, ->, bend right, "\text{Conj~\ref{conj: most general waist conj}}"' blue]
%\ar[ddd, to path={decorate{ -- ++(-2.5,0)   
%                -- ++(0,-4.7)   
%                -- ++(0.9,0)}}, decoration={snake, amplitude=0.3mm, segment length=1.5mm}]
\ar[r,Rightarrow]  
                & (\mathrm{T}_n)_{L^1} \ar[ull, decorate, decoration={snake, amplitude=0.3mm, segment length=1.5mm}, ->, bend right=8, "\text{Conj \ref{conj:TZimmer}}"' blue] \ar[u, 
                Rightarrow,"\text{Thm~\ref{thm:Farb2}}"' blue]
                \ar[d, Rightarrow,"\text{Thm~\ref{thm: higher prop T implies coboundary expansion}}" blue] \\
&\shortstack{\text{$\bbR$-coboundary} \\ \text{expansion and}\\\text{vanishing $H^\ast({\_},\bbR)$}\\\text{in degrees up to~$n$}}\ar[r, Rightarrow]\ar[d, dotted, Rightarrow, "\text{for $n=1$}" blue]      & \shortstack{\text{$\bbR$-coboundary} \\ \text{expansion}\\\text{up to~$n$}}\ar[d, Leftarrow]
%\ar[d, decorate, decoration={snake, amplitude=0.3mm, segment length=1.5mm}, ->,bend left=35] \\
\ar[d, decorate, decoration={snake, amplitude=0.3mm, segment length=1.5mm}, ->, bend right=65, "\shortstack{\text{\footnotesize{using}} \\\text{\scriptsize{Conj~\ref{conj: torsion cohomology conjecture}?}}}"' blue, near end] \\
\shortstack{\text{Waist inequalities}\\ \text{in codimension} \\ n+1}&\shortstack{\text{$\bbZ$-cosystolic} \\ \text{expansion}\\\text{up to~$n$}}\ar[d, Rightarrow]\ar[r, Rightarrow] \ar[l, Rightarrow,"\text{Thm~\ref{thm: topological overlap}}"' blue]         & \shortstack{\text{$\bbZ$-coboundary} \\ \text{expansion}\\\text{up to~$n$}}\ar[d, Rightarrow, dotted, "\shortstack{\text{\scriptsize{assume}} \\ \text{\scriptsize{$H^\ast({\_},\bbZ)=0$}}}" blue] \\
%& & & \\
&\shortstack{\text{$\bbF_2$-cosystolic} \\ \text{expansion}\\\text{up to~$n-1$}}\ar[r, Rightarrow] \ar[d, Rightarrow, "\text{Thm~\ref{thm: topological overlap}}" blue]    & \shortstack{\text{$\bbF_2$-coboundary} \\ \text{expansion}\\\text{up to~$n-1$}} \\
  &    \shortstack{\text{Topological overlap property}\\ \text{in dimension~$n$}} & &
\end{tikzcd}
\caption{Snake arrows are conjectural. Dotted arrows are conditional.}
\label{fig: geometric implications} \label{fig1}
\end{figure}
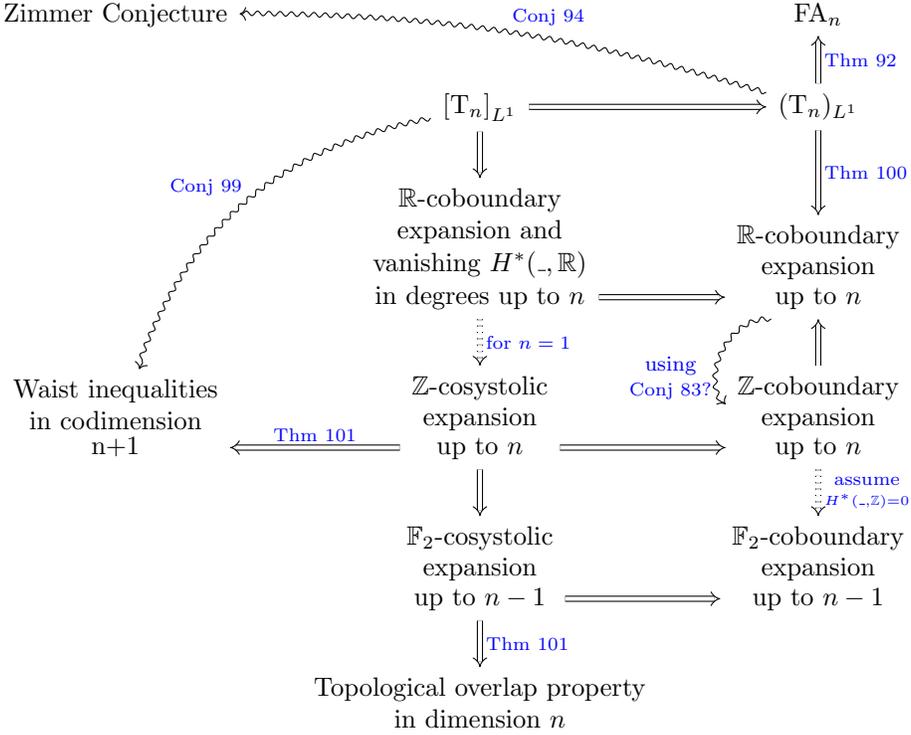

\begin{defn}
A family $\mathcal{F}$ of   finite volume Riemannian manifolds of the same dimension satisfies a \emph{uniform waist inequality in codimension~$d$}, where $d$ is smaller than the dimension, if there is a constant~$C_{\mathcal{F}}>0$ such that for every $M\in \mathcal{F}$ and for every smooth map $f\colon M\to \bbR^d$ there is a  regular value $x\in \bbR^d$ such that the $(n-d)$-dimensional volume of the fiber of~$x$ satisfies 
\[ \vol_{n-d}\bigl(f^{-1}(\{x\})\bigr)\ge C_{\mathcal{F}}\cdot\vol(M).\]
\end{defn}

Let us put this in contrast with the notion of topological expander. A family $\calF$ of simplicial complexes is a $d$-dimensional \emph{topological expander} if there is $C_\calF>0$ such that for every $X\in\calF$ and every continuous map $f\colon X\to\bbR^d$ there is a fiber $f^{-1}(\{x\})$ that meets at least a $C_\calF$-fraction of all  $d$-simplices in~$X$. 

\begin{conjecture}\label{conj: waist conjecture for lattice}
 Let $G$ be a simple Lie group of higher rank $d\ge 2$ and with finite center. Let $X=G/K$ be the associated symmetric space. Then the family of finite volume locally symmetric spaces with universal cover~$X$ satisfies a uniform waist inequality in codimensions $1,\dots, d$. 
\end{conjecture}

Using geometric measure theory and the theory of minimal surfaces, the following result was proved by  Fraczyk and Lowe~\cite{fraczyklowe}.

\begin{theorem}
The family of compact locally symmetric spaces whose universal cover is the octonionic hyperbolic plane satisfies a uniform waist inequality in codimension~$2$.
\end{theorem}

With completely different methods, the authors prove the following result. 

\begin{theorem}[\cite{BS-Waist}]\label{thm: prop T and waist}
The family of finite covers of a connected closed Riemannian manifold whose fundamental group has property~T satisfies a uniform waist inequality~\footnote{The allowed family of maps has to be restricted to real-analytic maps.} in codimension~$2$. 
\end{theorem}

In a very recent preprint, Cameron Gates Rudd~\cite{rudd} uses ideas from~\cite{BS-Waist} to prove that a residually finite $3$-dimensional Poincar\'e duality group cannot be K\"ahler. 

In view of the last theorem, one is tempted to make the following conjecture, which seems even more ambitious than  Conjecture~\ref{conj: waist conjecture for lattice}. 

\begin{conjecture}\label{conj: most general waist conj}
Let $M$ be a closed Riemannian manifold such that the fundamental group of~$M$ satisfies~$(T_n)_{L^1}$ and $M$ is $(n-1)$-connected. Does the family of finite covers of~$M$ satisfy a uniform waist inequality in codimension~$n$?     
\end{conjecture}

The proof of Theorem~\ref{thm: prop T and waist} relies on the concept of \emph{coboundary expansion} with respect to cellular $\ell^1$-norm in the simplicial or cellular cochain complexes -- a concept first studied by Linial-Meshulam. 
We call a CW-complex $X$ an \emph{$\bbR$-valued $\epsilon$-coboundary expander} in degree~$k$ if for every cochain  $g$ in the image of the differential $d\colon C^k(X;\bbR)\to C^{k+1}(X;\bbR)$ there is $f\in C^k(X;\bbR)$ such that $d(f)=g$ and $\Vert g\Vert \ge \epsilon\cdot \Vert f\Vert$. Obvious variations yield the definitions of \emph{$\bbZ$-valued or $\bbF_2$-valued $\epsilon$-coboundary expander}. 
A $\bbZ$-valued $\epsilon$-coboundary expander in degree~$k$ such that each integral $k$-cocycle that is non-zero in cohomology has norm at least~$\mu$ is called a \emph{$\bbZ$-valued $(\epsilon, \mu)$-cosystolic expander}. Compare with~\cite{evra+kaufman}*{Definition~1.7}. 

For the proof of Theorem~\ref{thm: prop T and waist} one fixes a triangulation of the base manifold~$M$ and endows every element of the family $\calF$ consisting of the finite covers of~$M$ with the lifted triangulations. 
The first step is to apply the following theorem to every triangulated manifold in~$\calF$ and for $k\in\{0,1\}$. 
The case $k\in\{0,1\}$ is proved in~\cite{BS-Waist} but there is a similar proof in arbitrary degrees. 

\begin{theorem}\label{thm: higher prop T implies coboundary expansion}
Let $X$ be a compact $(n-1)$-connected CW-complex such that 
its fundamental group satisfies property $(T_n)_{L^1}$. Then there is an $\epsilon>0$ so that every finite cover $\bar X$ of~$X$ is a $\bbR$-valued $\epsilon$-coboundary expander in degrees~$0,\dots, n$. 
\end{theorem}

Real-valued coboundary expansion is not sufficient to yield geometric implications. For that one has to switch to $\bbF_2$-valued or $\bbZ$-valued cosystolic expansion.

\begin{theorem}[Topological overlap property and waist inequalities]\label{thm: topological overlap}\hfil
\begin{enumerate}
\item Assume that there is $\epsilon>0$ and $\mu>0$ such that the family~$\calF$ of finite covers of a compact simplicial complex consists of $\bbF_2$-valued $(\epsilon,\mu)$-cosystolic expanders in degrees $0,1,\dots, d-1$. 

Then there is a uniform constant $c>0$ so that the $d$-skeleton of every complex in $\calF$ has Gromov's $c$-topological overlap property. 
\item Assume that there is $\epsilon>0$ and $\mu>0$ such that the family~$\calF$ of finite covers of a closed Riemannian manifold consists of $\bbZ$-valued $(\epsilon,\mu)$-cosystolic expanders\footnote{with respect to triangulations that are lifted from a triangulation of the base manifold.} in degrees~$0,1,\dots, d-1$. 

Then $\calF$ satisfies a uniform waist inequality in codimension~$d$. 
\end{enumerate}
\end{theorem}

We show in~\cite{BS-Waist}*{Theorem~2.12} for a compact CW-complex $X$ with $H^1(X,\bbZ)=0$ that $\bbR$-valued $\epsilon$-coboundary expansion implies $\bbZ$-valued $\epsilon$-cosystolic expansion in degrees~$0$ and~$1$. Together with Theorem~\ref{thm: topological overlap}, one concludes the proof of Theorem~\ref{thm: prop T and waist}. 
Our tool for extracting $\mathbb{Z}$-coboundary expansion out of the $\mathbb{R}$-coboundary expansion given in Theorem~\ref{thm: higher prop T implies coboundary expansion} is the \emph{total unimodularity} of certain coboundary maps, see \cite[Theorem 2.8]{BS-Waist}.
Without overly expanding on this at this point, we remark on the relation between total unimodularity of the boundary maps and homological torsion, which was established in \cite[Theorem 5.2]{THK}. It is also implicit in the recent preprint~\cite{torsion_jakob}. 
We cannot help speculating that one might combine $(T_n)_{L^1}$ and the sublinearity of torsion growth from Conjecture~\ref{conj: torsion cohomology conjecture} to establish $\mathbb{Z}$ cohomological isoperimetric inequalities for lattices.

One might wonder under which conditions one can deduce $\bbF_2$-valued coboundary expansion from $\bbZ$-valued coboundary expansion.
Here is a simple lemma showing a relationship. 

\begin{lemma} \label{lem:ZtoFp}
If $X$ is a finite simplicial complex that is a $\bbZ$-valued $\epsilon$-coboundary expander in degrees $0, \dots, k+1$ and $H^i(X,\bbZ)=0$ for each $i\in\{1,\dots, k+1\}$, then $X$ is a $\bbF_p$-valued $\epsilon'$-coboundary expander in degrees~$0,\dots, k$ for some $\epsilon'=\epsilon'(\epsilon, k)>0$.     
\end{lemma}

\begin{proof}[Sketch of proof]
    %For $k=0$ this follows from~\cite{}. Let $k\ge 1$. 
    Let $f\in C^k(X,\bbF_p)$ be a coboundary and lift it minimally to $\tilde f\in C^k(X,\bbZ)$. If $\tilde f$ were a cocycle, we could choose a minimal $g\in C^{k-1}(X,\bbZ)$ according to the assumption and project it down to $C^{k-1}(X,\bbF_p)$ to obtain a witness to the $\bbF_p$-valued coboundary expansion in degree~$k$. If $\tilde f$ is not a cocycle, we correct it by $p$ times a minimal primitive of $g\in C^{k+1}(X,\bbZ)$ where $g$ is the cocycle obtained from the cochain description of the boundary map in the Bockstein long exact sequence associated with $\bbZ\xrightarrow{\cdot p}\bbZ\to\bbF_p$. 
\end{proof}

In view of the above results, we predict the following. 

\begin{conjecture}
    Let $X$ be a connected $2$-dimensional simplicial complex whose fundamental group has property~T. Then there is $c>0$ such that every finite cover of $X$ has the $c$-topological overlap property. 
\end{conjecture}

Let us turn again to Conjectures~\ref{conj: waist conjecture for lattice} and~\ref{conj: most general waist conj}. Using Corollary~\ref{cor:com} we can apply  
Theorem~\ref{thm: higher prop T implies coboundary expansion} in both cases to get $\bbR$-valued coboundary expansion.  What is still missing is the transition to $\bbZ$ and from $\bbZ$-valued coboundary expansion to $\bbZ$-valued cosystolic expansion. 

\begin{question}
Under which conditions can one deduce  
$\bbZ$-valued coboundary expansion from $\bbR$-valued cosystolic expansion? Are there combinatorial-homological methods, reminiscent of Evra-Kaufman's local-to-global theorem~\cite{evra+kaufman}, that allow to further deduce $\bbZ$-valued cosystolic expansion? 
\end{question}

\subsection{Applications and conjectures in low degrees } \label{sec:sd}

\subsubsection{Degree 1: The spectral gap conjecture} \label{sec:deg1}

We will not review here the enormous variety of applications of property T and its variants.
Rather, we will focus on the bit which we understand less, that is,  
the degree~$1$ case of Conjecture~\ref{conj:semisimple}. Throughout, we will assume that $\Gamma$ is a center-free irreducible lattice in an $S$-semisimple group~$G$ of higher rank with no compact factors, and refer to $\Gamma$ just as a \emph{center-free irreducible higher-rank lattice}. 
%For simplicity of the discussion, our standing assumption in this subsection is that $G$ is a semisimple Lie group.
%For S-semisimple groups one needs to add Conjecture~\ref{conj: S-arithmetic filling}, which regards an analogue of a theorem by Leuzinger-Young in this setting, see Theorem~\ref{thm: leuzinger+young}.
The above conjecture is open, most intriguingly, already for unitary representations, in which case it is equivalent to the following conjecture of Shalom.

\begin{conjecture}[The spectral gap conjecture] \label{conj:SG}
    Let $\Gamma$ be a center-free irreducible higher rank lattice in $G$.
    Let $V$ be a unitary representation of $\Gamma$. Assume that the $\Gamma$-action on $V$ does not extend to a $G$-action on any non-trivial subrepresentation. Then $V$ has a spectral gap.
\end{conjecture}

\begin{proof}[Conjecture~\ref{conj:SG} $\Leftrightarrow$ Conjecture~\ref{conj:semisimple} in degree 1 for unitary representations]
    Let $V$ be a unitary representation of $\Gamma$.
    If $V$ does not have a spectral gap,  then $H^1(\Gamma,V)$ is non-Hausdorff and in particular non-zero. So in this case Conjecture~\ref{conj:semisimple} implies the existence of a non-trivial subrepresentation of $V$ on which the $\Gamma$-action extends to a $G$-action.
    For the reverse direction, assume that $H^1(\Gamma,V)\neq 0$  and that the $\Gamma$-action on $V$ does not extend to a $G$-action on any non-trivial subrepresentation. We now consider the unitary induction $I(V)$ of $V$ from $\Gamma$ to $G$. We have $H^1(G,I(V))\neq 0$. If $\Gamma<G$ is uniform, this follows from the Shapiro lemma. If $\Gamma<G$ is non-uniform, one can use the replacement of the Shapiro lemma mentioned in~\S\ref{subsec:poly}. 
    By \cite[Proposition~5.1]{BBHP}, we have that for every factor $G_1<G$, $I(V)^{G_1}=0$. We conclude by a result of Shalom~\cite{Shalom}*{Proposition~3.2} that $H^1(G,I(V))$ is non-Hausdorff. Thus $I(V)$ does not have a spectral gap.
    Therefore, $V$ does not have a spectral gap.
\end{proof}

\begin{proof}[$G$ has a non-compact factor with T $\Rightarrow$ Conjecture~\ref{conj:SG}]
    Consider a non-compact factor $G_1<G$ with property~T.
    Assume that $V$ does not have a spectral gap.
    Then its unitary induction $I(V)$ to $G$ does not have a spectral gap as $G$-representation. Hence also not as a $G_1$-representation. We deduce that $I(V)^{G_1}\neq 0$.
    However, by 
    \cite[Proposition 5.1]{BBHP} $I(V)$ could be identified with a subrepresentation of $V$ on which the $\Gamma$-action extends to a $G$-action.
\end{proof}

The following is a stronger form of Conjecture~\ref{conj:SG}.

\begin{conjecture}[The strong spectral gap conjecture] \label{conj:SSG}
    Let $\Gamma$ be a center-free irreducible higher rank lattice in $G$.
    Let $V$ be a unitary representation of $\Gamma$ and assume that the $\Gamma$-action on $V$ does not extend to a $G$-action on any non-trivial subrepresentation. Then the unitary induction of $V$ to $G$ has a spectral gap with respect to each factor of $G$.    
\end{conjecture}

Clearly, Conjecture~\ref{conj:SSG} implies Conjecture~\ref{conj:SG} 
and it holds if $G$ has property~T.
If $G$ is a Lie group, the argument of~\cite[Theorem D]{badsau} shows the
theorem still holds without assuming property~T - provided Conjecture~\ref{conj:SSG}, see \cite[Conjecture 1.5]{badsau}.
That is, we have: 

\vskip1mm

\emph{In the Lie case we have the implication:  Conjecture~\ref{conj:SSG} $\Rightarrow$ Conjecture~\ref{conj:semisimple} for unitary representations}.

\vskip1mm

This extends verbatim to S-semisimple groups and cocompact lattices, in particular to $p$-adic groups, where all lattices are cocompact.
For non-cocompact lattices in S-semisimple groups, the implication above still holds, assuming Conjecture~\ref{conj: S-arithmetic filling}.

We list below various conjectures regarding irreducible higher rank lattices
which are all implied by 
Conjecture~\ref{conj:SG} and in particular by Conjecture~\ref{conj:semisimple}.
For convenience, see Figure~\ref{fig2} for a summary of the relation between the various conjectures discussed in this subsection.

\begin{figure}[ht] 
\centering
\begin{tikzcd}[scale=0.8]
& \mbox{\textcolor{blue}{$G$ has T}} \ar[r,Rightarrow] \ar[d,Rightarrow] & \mbox{\textcolor{blue}{$G$ has a T factor}} \ar[d,Rightarrow] & \\
& \shortstack{\mbox{Strong} \\ \mbox{Spectral Gap}} \ar[r,Rightarrow] \ar[d,Rightarrow, "\shortstack{\mbox{\scriptsize{for S-semisimple}}\\\mbox{\scriptsize{assuming}}\\ \mbox{\scriptsize{Conjecture~\ref{conj: S-arithmetic filling}}}}"' blue] & \mbox{Spectral gap} \ar[d,Leftrightarrow] \ar[r,Rightarrow] & \mbox{Property } \tau \\
\mbox{Conjecture~\ref{conj:semisimple}} \ar[r,Rightarrow] & \shortstack{\text{Conjecture~\ref{conj:semisimple}} \\\text{for unireps}} \ar[r,Rightarrow] \ar[d,Rightarrow] &  \shortstack{\mbox{Conjecture~\ref{conj:semisimple}}\\ \mbox{for unireps} \\ \mbox{in degree 1}}\ar[d,Rightarrow] \ar[r,Rightarrow] & \shortstack{\mbox{coamenable} \\ \mbox{subgroups}} \ar[d,Rightarrow] \\
& \shortstack{\mbox{URS} \\ \mbox{rigidity}} \ar[r,Leftarrow] & \mbox{Charfiniteness} \ar[r,Rightarrow] \ar[d,Leftrightarrow] & \shortstack{\mbox{IRS} \\ \mbox{rigidity}} \\
&  & \mbox{Character Rigidity} & 
\end{tikzcd}
\caption{Relations between conjectures for $\Gamma<G$ irreducible higher rank lattice. Extra assumptions are in blue.}
\label{fig2}
\end{figure}
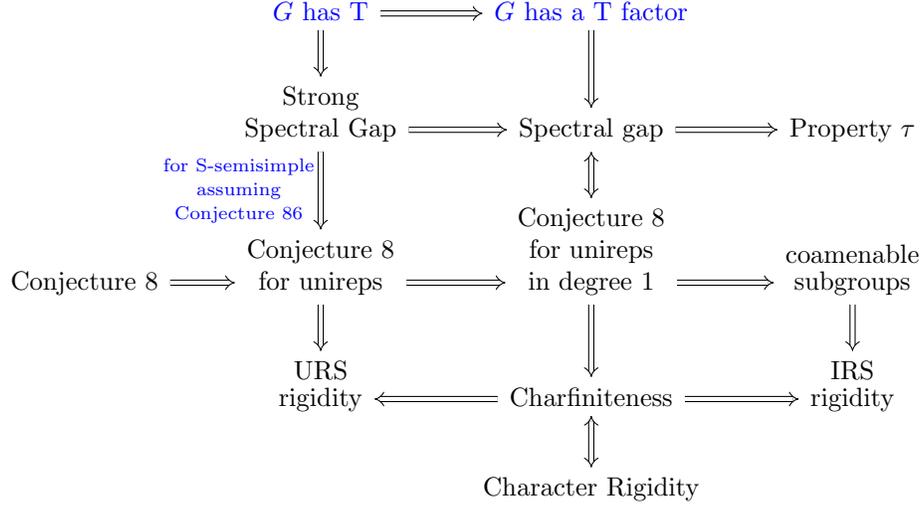

We start with the following conjecture which is due to Lubotzky.

\begin{conjecture}[Property $\tau$] \label{conj:tau}
    Let $\Gamma$ be an irreducible higher rank lattice.
    Then $\Gamma$ has property $\tau$.
\end{conjecture}

\begin{proof}[Conjecture~\ref{conj:SG} $\Rightarrow$ Conjecture~\ref{conj:tau}]
 Let $\hat{\Gamma}$ be the profinite completion. By Howe-Moore theorem the $\Gamma$-action on $L^2(\hat{\Gamma})$ does not extend to a $G$-action on any non-trivial subrepresentation of $L^2(\hat{\Gamma})$. We deduce that $\Gamma$ has property $\tau$.
\end{proof}

\vskip-1mm

The following conjecture is due to Cornulier.
Some progress in this direction was recently made in \cite{BGL}.

\begin{conjecture}[Coamenable subgroups] \label{conj:coam}
    Let $\Gamma$ be an irreducible higher rank lattice.
    Then every coamenable subgroup in $\Gamma$ is of finite index.
\end{conjecture}

\begin{proof}[Conjecture~\ref{conj:SG} $\Rightarrow$ Conjecture~\ref{conj:coam}]
    Let $\Lambda<\Gamma$ be a coamenable subgroup. Thus $\ell^2(\Gamma/\Lambda)$ does not have a spectral gap. Let $V_0\subset \ell^2(\Gamma/\Lambda)$ be a non-trivial subspace on which the $\Gamma$-action extends to~$G$.
    Let $v\in V_0$ be the image of the cyclic vector $\delta_\Lambda$ under the orthogonal projection $\ell^2(\Gamma/\Lambda)\to V_0$ and note that it is $\Lambda$-invariant and cyclic, hence non-zero. The group $\Lambda$ is infinite because $\Gamma$ is non-amenable. By the Howe-Moore theorem, $v$ is $G$-invariant,  in particular $\Gamma$-invariant. It follows that $\Gamma/\Lambda$ is finite.
\end{proof}

The following is a conjecture of Stuck-Zimmer~\cite{stuck+zimmer}.

\begin{conjecture}[IRS rigidity] \label{conj:SZ}
Let $\Gamma$ be a center-free irreducible higher rank lattice.
Then for every ergodic probability measure preserving action on a probability space, either the space is essentially finite or the action is essentially free.    
\end{conjecture}

\begin{proof}[Conjecture~\ref{conj:coam} $\Rightarrow$ Conjecture~\ref{conj:SZ}]
It follows from the work of Stuck and Zimmer that if the action is not essentially free then a.e stabilizer is coamenable, hence a.e orbit is finite. By ergodicity the space is essentially finite.
\end{proof}

Let us explain the term \emph{IRS rigidity}.
An \emph{invariant random subgroup} of $\Gamma$, in short an \emph{IRS}, is a conjugation invariant probability measure on $\Sub(\Gamma)$, the space of subgroups of $\Gamma$, endowed with the Chabauty topology.
An equivalent formulation of Conjecture~\ref{conj:SZ} is that every IRS on $\Gamma$ is supported on a finite set.
Similarly, a URS on $\Gamma$ is a minimal subsystem in $\Sub(\Gamma)$.

\begin{conjecture}[URS rigidity] \label{conj:URS}
    Let $\Gamma$ be an irreducible higher rank lattice.
    Then every URS on $\Gamma$ is supported on a finite set.
\end{conjecture}

In \cite{BBHP} the concepts of charmenability and charfiniteness are defined.
We will not repeat the definitions here.
It is shown there that every irreducible higher rank lattice is charmenable and conjectured that it is charfinite.

\begin{conjecture}[charfiniteness] \label{conj:charfinite}
    Let $\Gamma$ be an irreducible higher rank lattice.
    Then $\Gamma$ is charfinite.
\end{conjecture}

Before this we need some preparation, taken from~\cite{BSS}.
Let $G$ be a semisimple group.
For every unitary representation $V$ of $G$ 
we let $V_{\temp}$ be the unique maximal subrepresentation that is tempered (that is, weakly contained in the regular representation). We let $V_{\rig}$ be its orthogonal complement, thus 
$V=V_{\temp}\oplus V_{\rig}$.
If $V=V_{\rig}$ we say that the representation $V$ is of \emph{rigid type}.
The following theorem explains why.

\begin{theorem}[\cite{BSS}] \label{thm:reg-rig}
Let $G$ be a connected semisimple Lie group with finite center. 
Let $\Gamma<G$ be an irreducible lattice.
Let $U$ and $V$ be unitary $G$-representations and let $T\colon U\to V$ be a $\Gamma$-morphism,
that is, a bounded linear map that is $\Gamma$-equivariant.
Then we have $T(U_{\temp})\subset V_{\temp}$ and $T(U_{\rig})\subset V_{\rig}$.
Furthermore, $T$ is $G$-equivariant on $U_{\rig}$.
\end{theorem}

\begin{theorem}[\cite{BSS}] \label{thm:rig}
Let $G$ and $\Gamma$ be as in Theorem~\ref{thm:reg-rig}. 
Let $\Gamma<G$ be an irreducible lattice.
Let $U$ be a unitary representation of $\Gamma$.
Then there exists a unique maximal subrepresentation $U_{\rig}$ of $U$ on which the $\Gamma$-action extends to a unitary $G$-representation of rigid type.
\end{theorem}

Now we are able to prove the following implication. 

\begin{proof}[Conjecture~\ref{conj:SG} $\Rightarrow$ Conjecture~\ref{conj:charfinite}]

By the main results of \cite{BBHP} and \cite{BH19}, $\Gamma$ is charmenable,
thus it is charfinite if it satisfies items (3)-(5) in \cite[Definition~1.2]{BBHP}.
Item (3) is clear by Zariski density and (4) follows from super-rigidity. So we only need to establish (5), that is, to show that every amenable extremal character of~$\Gamma$ is finite.
To this end, let us fix an amenable extremal character $\phi$. 

We let $\pi\colon\Gamma\to \mathcal{U}(V)$ be the associated amenable GNS representation and $v\in V$ be the corresponding cyclic vector.
Then there exists a representation $\rho\colon\Gamma\to \mathcal{U}(V)$ such that $\rho(\Gamma)$ commutes with $\pi(\Gamma)$, giving rise to a representation $\pi\times \rho:\Gamma\times \Gamma\to \mathcal{U}(V)$ such that $v$ is $(\pi\times \rho)(\Delta)$-invariant, where $\Delta<\Gamma\times \Gamma$ is the diagonal subgroup. 
By the equation 
\[ \phi(g)=\langle \pi(g)v,v \rangle=\langle \pi(g)(\pi\times \rho)(g^{-1},g^{-1})v,v \rangle=\langle \rho(g^{-1})v,v \rangle=\langle \rho(g)^*v,v \rangle \]
and the fact that $v$ is $\pi$-cyclic, we conclude that $\rho\cong \bar{\pi}$ and $v$ is  $\rho$-cyclic.

Similarly, we have the representation $\bar{\pi}\colon\Gamma\to \mathcal{U}(V)$, associated with the character $\bar{\phi}$, and the corresponding commuting representation $\bar{\rho}$.
We get a representation 
\[ \pi\times \bar{\pi}\times \rho\times \bar{\rho}\colon\Gamma\times\Gamma\times\Gamma\times\Gamma\to \mathcal{U}(V\hat{\otimes} V). \]
The vector $v\otimes v$ is cyclic for both $\pi\times \bar{\pi}$
and $\rho\times \bar{\rho}$.
We consider the representation $\pi\otimes \bar{\pi}=(\pi\times \bar{\pi})|_\Delta$
and let $U\subset V\hat{\otimes} V$ be the cyclic $\pi\otimes \bar{\pi}$-subrepresentation generated by $v\otimes v$.
The GNS representation associated with $|\phi|^2=\phi\cdot\bar{\phi}$ is isomorphic to $U$,
endowed with the representation $\pi\otimes \bar{\pi}$ and the commuting representation $\rho\otimes \bar{\rho}$, where the vector $v\otimes v$ is invariant under the corresponding diagonal
and is cyclic for both representations.
By the amenability of $\pi$ we obtain that $\pi\otimes \bar{\pi}$ has almost invariant vectors in $V\hat{\otimes} V$. We claim that it already has almost invariant vectors in~$U$.

We now prove the claim.
We let $\mu\in \Prob(\Gamma)$ be the uniform measure on a finite set of generators and we consider the associated averaging operator $T=\pi\otimes \bar{\pi}(\mu)$.
This is a positive self adjoint operator on $V\hat{\otimes} V$ which preserves $U$.
It has spectrum in $[0,1]$.
For every $\epsilon>0$ we consider the associated $[1-\epsilon,1]$-spectral projection $p_\epsilon$,
which also preserves $U$.
Since $\pi\otimes \bar{\pi}$ has almost invariant vectors, we conclude that $p_\epsilon\neq 0$.
By the fact that $v\otimes v$ is $\rho\times \bar{\rho}$-cyclic
and $\rho\times \bar{\rho}(\Gamma\times \Gamma)$ commutes with $T$, we conclude that $p_\epsilon(v\otimes v)\neq 0$.
We note that $p_\epsilon(v\otimes v)\in U$.
Since $\epsilon$ was arbitrary, this shows that~$1$ is in the spectrum of~$T$ as an operator on~$U$.
This proves the claim.

We keep considering  $U$ as a $\pi\otimes \bar{\pi}$ representation
and let $U_{\rig}<U$ be the unique maximal subrepresentation on which the $\Gamma$-action extends to a unitary $G$-representation of rigid type, guaranteed by Theorem~\ref{thm:rig}.
We note that $U_{\rig}$ is $\rho\otimes \bar{\rho}$-invariant, by commutation.

Next we show that the complement $U_{\rig}^\perp$ of $U_{\rig}$ in $U$ has a spectral gap: A maximal subrepresentation of $U_{\rig}^\perp$ on which the $\Gamma$-action extends to a unitary $G$-representation (which exists by Zorn's Lemma) must be tempered. Hence it has a spectral gap even when restricted to a simple factor $G_1<G$, as does $L^2(G_1)$.  Further, its complement in $U_{\rig}^\perp$ has a spectral gap by the assumed Conjecture~\ref{conj:SG}.

So the subrepresentation  $U_{\rig}$ is non-trivial because $\pi\otimes \bar{\pi}$ has almost invariants vectors in $U$ and $U_{\rig}^\perp$ has a spectral gap. 

We consider the orthogonal projection from $U$ to $U_{\rig}$, which is $(\pi\otimes \bar{\pi})\times (\rho\otimes \bar{\rho})$ equivariant, and denote by $u<U_{\rig}$ the image of $v\otimes v$.
This is a cyclic vector in $U_{\rig}$ for both $\pi\otimes \bar{\pi}$ and $\rho\otimes \bar{\rho}$,
in particular $u\neq 0$,
and it is stable under the diagonal subgroup $((\pi\otimes \bar{\pi})\times (\rho\otimes \bar{\rho}))(\Delta)$.
We conclude as before that $(\rho\otimes \bar{\rho})|_{U_{\rig}}$ is the complex conjugate of $(\pi\otimes \bar{\pi})|_{U_{\rig}}$. 
In particular, $(\rho\otimes \bar{\rho})|_{U_{\rig}}$ extends to a $G$-representation on $U_{\rig}$.
As we may, we view $U_{\rig}$ as a $G\times G$-representation.

The vector $0\neq u\in U_{\rig}$ is $\Delta$-invariant, thus $G\times G$-invariant by the Howe-Moore Theorem, and in particular $\Gamma\times \Gamma$-invariant.
It follows that $V$ is not $\pi\times \rho$-weakly mixing. Thus it contains a finite dimensional $\Gamma\times \Gamma$-subrepresentation, which equals $V$ itself by the extremality assumption on $\phi$.
Hence $\phi$ is finite
\end{proof}

\begin{proof}[Conjecture~\ref{conj:charfinite} $\Rightarrow$ Conjecture~\ref{conj:SZ} and Conjecture~\ref{conj:URS}]
See \cite[Propositions 3.4 and 3.5]{BBHP}.
\end{proof}

The following conjecture is due to Connes.
It is a far reaching generalization of Margulis' Normal Subgroup Theorem.
It was proved by Bekka for $\SL_n(\mathbb{Z})$, $n\geq 3$ and by Peterson in case all simple factors of $G$ are of higher rank.

\begin{conjecture}[Character Rigidity] \label{conj:CR}
    Every character of a higher rank lattice is either central or finite dimensional.
\end{conjecture}

\begin{proof}[Conjecture~\ref{conj:charfinite} $\Leftrightarrow$ Conjecture~\ref{conj:CR}]
See~\cite{BBHP}.
\end{proof}

In case $G$ has a factor with T,
Conjecture~\ref{conj:charfinite} was proved in~\cite{BBHP} and~\cite{BH19}.
Very recently it was proved with no conditions on~$G$ for all non-uniform lattices by Dogon-Glasner-Gorfine-Hanany-Levit~\cite{5authors}. 
However, it is still open for all cocompact lattices in products of groups which do not have T.

\subsubsection{Degree 2: Stability}

An excellent background reference for stability is Thom's article~\cite{Thom}.
Here we will only make some comments on recent developments. 

Cohomological vanishing in  degree 2 has applications in the theory of stability. This was an important insight of the work of de Chiffre-Glebsky-Lubotzky-Thom~\cite{CGLT}. 
Based on this and Garland's work they constructed central extensions of certain higher rank $p$-adic lattices that are not Frobenius-approximable. Based on \cite{CGLT} and on~\cite{badsau}, Bader-Lubotzky-Sauer-Weinberger showed that the non-trivial central extensions of symplectic lattices are not Frobenius-approximable~\cite{bader+lubotzky+sauer+weinberger}.

\cite[Question 3.10]{Thom} was answered positively in \cite{badsau}, see Theorem~\ref{thm:higher}.
\cite[Question 3.11]{Thom} was answered positively for $1<p<\infty$ in Lubotzky-Oppenheim \cite{Lub-Opp},
using vanishing of second cohomology in certain $L^p$-representations to show Schatten-$p$-norm stability and non-approximability.
For every $1<p_1<p_2<\infty$ they constructed groups with vanishing of second cohomology with coefficients in these $L^p$ representations, for every $p\in[p_1,p_2]$.
They used that to construct a finitely presented group which is not $p$-norm approximated for any $1<p<\infty$.
They notably left open the cases $p=1$ and $p=\infty$, both with respect to cohomological vanishing and non-approximability. 
The non-approximability question in $p=1$ was very recently solved by Bachner-Dogon-Lubotzky in \cite{BDLub}, in a way bypassing the cohomological vanishing question.
Corollary~\ref{cor:com} settles completely the question of cohomological vanishing also in the case $p=1$.
The case $p=\infty$ remains open. We note, however, that the discouraging reasoning given by the end of the introduction of \cite{Lub-Opp}, describing why their methods must be insufficient for this case, seems somewhat less discouraging in view of Theorem~\ref{thm:vNcriterion} and Conjecture~\ref{conj:tracial} which regard vanishing of cohomology with von Neumann algebra coefficients.

\subsubsection{Degree 2: Applications to measurable diversity}\label{subsec:measdiv}

The fact that there are hyperbolic groups, like uniform lattices in the isometry group of the Cayley hyperbolic plane, with highter property~T (see Theorem~\ref{thm:F4}) combined with the theory of group-theoretic Dehn fillings \cites{osin, DGO}, especially work of Petrosyan--Sun~\cites{PS1, PS2, PS3}, is an interesting source of tension between rigidity and flexibility phenomena. 
In~\cite{fournierfacio+sauer} Fournier-Facio and the second author use the cohomological excision theorem for Cohen-Lyndon triples and property $(T_2)$ to obtain a good control on the second cohomology of hyperbolic quotients obtained by group-theoretic Dehn fillings. See also~\cite{fournierfacio} for a similar construction that shows that poperty $(T_2)$ does not pass to quotients in general. The rigid control on the second cohomology together with the flexibility of group-theoretic Dehn fillings allows to construct an uncountable family consisting of colimits of quotients of a fixed hyperbolic group with property $(T_2)$ such that the range of the second $\ell^2$-Betti number is uncountable. More precisely, the following can be shown if one takes a uniform lattice in the isometry group of the Cayley hyperbolic plane as a starting point.

\begin{theorem}[\cite{fournierfacio+sauer}*{Theorem~A}]
There exists an uncountable family of simple Kazhdan groups of rational cohomological dimension~$16$ with pairwise distinct positive second $\ell^2$-Betti number. Furthermore, these groups are pairwise non-measure equivalent. 
\end{theorem}
It is an easy consequence of Gaboriau's construction \cite{gaboriau:cost}*{Proposition VI.16} that there exist uncountably many pairwise non-measure equivalent groups \cite{ITD}*{Proposition 5.1}. The previous theorem is the first construction of a measurably diverse family with bounded rational cohomological dimension. 
Two constructions of measurably diverse families of Kazhdan groups have appeared last year and precede the above theorem. 
One is due to Fournier-Facio and Sun \cite{dimensions}, where the groups are distinguished by their entire sequence of $\ell^2$-Betti numbers, in a spirit similar to the work of L\'opez Neumann \cite{antonio}, who constructs an infinite family of finitely presented simple groups. Very recently, Ioana--Tucker-Drob \cite{ITD} produced further examples of groups with vanishing $\ell^2$-Betti numbers, arising from the theory of wreath-like products~\cite{wreathlike}. In particular, these groups have infinite rational cohomological dimension.

%\section{tensor f.d}

\bibliographystyle{amsalpha} % or amsralpha, amsplain, amsalpha
\bibliography{biblio.bib}

\end{document}